\documentclass[11pt,reqno]{amsart}
\usepackage{fullpage,times,graphicx,amssymb,amsmath,psfrag,xcolor}
\definecolor{ddarkbrown}{rgb}{0.5,0.2,0.05} \definecolor{bbluegray}{rgb}{0.05,0,0.5}
\usepackage[colorlinks,citecolor=bbluegray,linkcolor=ddarkbrown,urlcolor=blue,breaklinks]{hyperref}
\usepackage{auto-pst-pdf}
\usepackage[round]{natbib}
\newtheorem{theorem}{Theorem}[section]
\newtheorem{proposition}[theorem]{Proposition}
\newtheorem{lemma}[theorem]{Lemma}

\newcommand{\BEAS}{\begin{eqnarray*}}
\newcommand{\EEAS}{\end{eqnarray*}}
\newcommand{\BEA}{\begin{eqnarray}}
\newcommand{\EEA}{\end{eqnarray}}
\newcommand{\BEQ}{\begin{equation}}
\newcommand{\EEQ}{\end{equation}}
\newcommand{\BIT}{\begin{itemize}}
\newcommand{\EIT}{\end{itemize}}
\newcommand{\BNUM}{\begin{enumerate}}
\newcommand{\ENUM}{\end{enumerate}}

\newcommand{\BA}{\begin{array}}
\newcommand{\EA}{\end{array}}

\newcommand{\refp}[1]{(\ref{#1})}

\newcommand{\ones}{\mathbf 1}

\newcommand{\reals}{{\mbox{\bf R}}}
\newcommand{\sreals}{\scriptsize{\mbox{\bf R}}}

\newcommand{\symm}{{\mbox{\bf S}}}  

\newcommand{\Rank}{\mathop{\bf Rank}}
\newcommand{\NumRank}{\mathop{\bf NumRank}}
\newcommand{\NumCard}{\mathop{\bf NumCard}}
\newcommand{\Card}{\mathop{\bf Card}}
\newcommand{\Tr}{\mathop{\bf Tr}}
\newcommand{\diag}{\mathop{\bf diag}}
\newcommand{\lambdamax}{{\lambda_{\rm max}}}

\newcommand{\idm}{\mathbf{I}}

\newcommand{\Expect}{\textstyle\mathop{\bf E{}}}
\newcommand{\Prob}{\mathop{\bf Prob}}

\newcommand{\var}{\mathop{\bf var}}

\newcommand{\redtext}[1]{\textcolor{red}{#1}}

\begin{document}
\title{Weak Recovery Conditions from Graph Partitioning Bounds\\ and Order Statistics}
\author{Alexandre d'Aspremont}
\address{CMAP, Ecole Polytechnique, UMR CNRS 7641.}
\email{alexandre.daspremont@m4x.org}
\author{Noureddine El Karoui}
\address{Statistics, U.C. Berkeley. Berkeley, CA 94720.}
\email{nkaroui@stat.berkeley.edu}

\date{\today}

\begin{abstract}
We study a weaker formulation of the nullspace property which guarantees recovery of sparse signals from linear measurements by $\ell_1$ minimization. We require this condition to hold only with high probability, given a distribution on the nullspace of the coding matrix $A$. Under some assumptions on the distribution of the reconstruction error, we show that testing these weak conditions means bounding the optimal value of two classical graph partitioning problems: the $k$-Dense-Subgraph and MaxCut problems. Both problems admit efficient, relatively tight relaxations and we use a randomization argument to produce new approximation bounds for $k$-Dense-Subgraph. We test the performance of our results on several families of coding matrices.
\end{abstract}
\maketitle

\section{Introduction}
Given a coding matrix $A\in\reals^{q\times n}$ and a signal $e\in\reals^n$, we focus on conditions under which the solution $x_0$ to the following minimum cardinality problem
\BEQ\label{eq:min-card}
\BA{rll}
x_0= & \text{min.} & \Card(x)\\
& \mbox{subject to} & Ax=Ae,\\
\EA
\tag{$\ell_0$-recov.}
\EEQ
which is a combinatorial problem in $x\in\reals^n$, can be recovered by solving
\BEQ\label{eq:min-ell1}
\BA{rll}
x^{\text{lp}} = &\mbox{min.} & \|x\|_1\\
& \mbox{subject to} & Ax=Ae,\\
\EA
\tag{$\ell_1$-recov.}
\EEQ
which is a convex program in $x\in\reals^n$. Problem~(\ref{eq:min-card}) arises in various fields ranging from signal processing to statistics. Suppose for example that we make a few linear measurements of a high dimensional signal, which admits a sparse representation in a well chosen basis (e.g. Fourier, wavelet). Under certain conditions, solving~(\ref{eq:min-ell1}) will allow us to reconstruct the signal exactly \citep{Dono04,Dono05,Dono06a}. In a coding application, suppose we transmit a message which is corrupted by a few errors, solving~(\ref{eq:min-ell1}) will then allow us to reconstruct the message exactly \citep{Cand05,Cand06}. Finally, problem (\ref{eq:min-ell1}) is directly connected to variable selection and penalized regression problems (e.g. LASSO \cite{Tibs96}) arising in statistics \citep{Zhao06,Mein07,Mein07a,Cand07,Bick07,Cand09a}. Of course, in all these fields, problems~(\ref{eq:min-card}) and~(\ref{eq:min-ell1}) are overly simplified. In practice for example, the observations could be noisy, approximate solutions might be sufficient and we might have strict computational limits on the decoding side. While important, these extensions are outside the scope of this work.

Based on results by \cite{Vers92} and \cite{Affe92}, \cite{Dono05} showed that when the solution $x_0$ of (\ref{eq:min-card}) is sparse with $\Card(x_0)=k$ and the coefficients of $A$ are i.i.d. Gaussian, then w.h.p. the solution of the (convex) problem in (\ref{eq:min-ell1}) will always match that of the combinatorial problem in~(\ref{eq:min-card}) provided $k$ is below an explicitly computable {\em strong recovery} threshold $k_S$. They also show that if $k$ is below another (larger) {\em weak recovery} threshold~$k_W$, then these solutions match with an exponentially small probability of failure.

Generic conditions for strong recovery based on sparse extremal eigenvalues, or {\em restricted isometry properties} (RIP), were also derived in \cite{Cand05} and \cite{Cand06}, who proved that certain random matrix classes satisfied these conditions near optimal values of $k$ with an exponentially small probability of failure. Simpler, weaker conditions which can be traced back to \cite{Dono01}, \cite{Zhan05} or \cite{Cohe06} for example, are based on properties of the nullspace of $A$. When the signal cardinality $\Card(e)\leq k$ and the {\em Nullspace Property} (NSP) holds, i.e. when there is a constant $\alpha_k<1/2$ such that
\BEQ\label{eq:ineq-a}
\|x\|_{k,1} \leq \alpha_k \|x\|_1 \tag{det-NSP}
\EEQ
for all vectors $x\in\reals^n$ with $Ax=0$, then solving the convex problem~(\ref{eq:min-ell1}) will recover the global solution to the combinatorial problem~(\ref{eq:min-card}). Condition~\eqref{eq:ineq-a} can be understood as an incoherence measure, i.e. it means that not all of the mass in $x$ can be concentrated among only $k$ coefficients, in other words:
\begin{center}
{\em Good coding matrices have incoherent nullspace vectors.}
\end{center}
In particular, this condition means that the nullspace of $A$ cannot contain sparse vectors. Furthermore, the constant $\alpha_k$ can be used to explicitly bound the reconstruction error when solving the $\ell_1$-recovery problem in~(\ref{eq:min-ell1}). This is illustrated in Proposition~\ref{prop:recon-error} below, directly adapted from~\citet[Th. 4.3]{Cohe06}.

One fundamental issue with the sparse recovery conditions described above is that, except for explicit thresholds available for certain types of random matrices (with high probability), testing these conditions on generic (deterministic) matrices is potentially {\em harder} than solving the combinatorial $\ell_0$-norm minimization problem in~\eqref{eq:min-card} for example as it implies either solving a combinatorial problem to compute $\alpha_k$ in~\eqref{eq:ineq-a}, or computing sparse eigenvalues. Recent results in \cite{Cand09a} show that the traditional (and tractable) incoherence conditions ensure recovery of sparse signals with high probability, given a uniform distribution on the signal. These incoherence conditions lack universality however, in the sense that contrary to the combinatorial conditions mentioned above, they cannot be used to guarantee recovery of {\em all} signals of near-optimal size $k$. Convex relaxation relaxation bounds were used in \citet{dAsp08b} (on sparse eigenvalues), \citet{Judi08} or \citet{dAsp08a} (on NSP) to test sparse recovery conditions similar to~\eqref{eq:ineq-a} on arbitrary matrices. Unfortunately, the performance (tightness) of these relaxations is still very insufficient: for matrices satisfying the sparse recovery conditions in \cite{Cand05} up to signal cardinality $k^*$, these three relaxations can only certify that the conditions hold up to cardinality $\sqrt{k^*}$ and are also likely to provide poor bounds on reconstruction error.

In what follows, we seek to enforce a weaker version of condition~(\ref{eq:ineq-a}). We will bound the incoherence measure $\alpha_k$ in~\eqref{eq:ineq-a} with high probability over a random sample of vectors in the nullspace of $A$. Another way to look at this approach is to remember that, if $x^\mathrm{lp}$ solves the $\ell_1$-decoding problem in~(\ref{eq:min-ell1}), the vector $x^\mathrm{lp}-e$ is always in the nullspace of $A$ and Proposition~\ref{prop:recon-error} below shows that enforcing condition~(\ref{eq:ineq-a}) on the reconstruction error $x^\mathrm{lp}-e$ allows us to bound the magnitude of this error.

Here, because we cannot efficiently test condition~\eqref{eq:ineq-a} over all vectors in the nullspace of $A$, we will instead require condition~(\ref{eq:ineq-a}) to hold only with high probability on the nullspace of $A$, given a distribution on this subspace. Let us assume for simplicity that $\Rank(A)=q$, and let $F\in\reals^{n\times m}$ with $m=n-q$ be a basis for the nullspace of $A$ (not necessarily orthogonal or normalized). We will require that the NSP condition (\ref{eq:ineq-a}) discussed above, which reads
\BEQ\label{eq:ineq-rnd}
\|Fy\|_{k,1} \leq \alpha_k \|Fy\|_1 \tag{proba-NSP}
\EEQ
be satisfied with high probability, given a distribution on $y$. We will start by assuming that $y$ is Gaussian. In this case, we will see that both sides of condition~(\ref{eq:ineq-rnd}) can be explicitly controlled by the solution of classic graph partitioning problems. These combinatorial problems admit tight, efficiently computable approximations which will allow us to bound the probability that~\eqref{eq:ineq-rnd} holds. We will then extend these results to more general distributions on the nullspace and show that the same quantities which controlled concentration in the Gaussian case, also control fluctuations in the more general model.

Of course, assuming the true distribution on the signal $e$ is either sparse or follows a power law, our simple model on the nullspace of $A$ error could have zero measure with respect to the true (structured) distribution of $x^\mathrm{lp}-e$, especially since $x^\mathrm{lp}$ is dependent on $A$. In fact, at first sight, we are implicitly positing a model on the reconstruction error, then ultimately use the model to bound this same reconstruction error, an apparent circular reference. Our main objective however is not to directly bound the error but rather to isolate efficiently computable quantities which will be good proxies for this error, sacrificing some statistical accuracy in favor of computational efficiency. Moreover, our main result is to efficiently approximate the Lipschitz constants of the two norms in~\eqref{eq:ineq-rnd}, constants which are likely to play a critical role {\em whatever the model} on the reconstruction error. Overall, the phase transition for signal recovery is usually very sharp (or ``bang-bang'') meaning that either all signals are recovered perfectly or none of them are. This means that choosing a realistic statistical model is probably not that crucial.

Current results in compressed sensing provide universal recovery guarantees using intractable conditions (which can only be tested with high probability on {\em random matrices}). Our objective here is to do the opposite and isolate {\em tractable} measures of performance that can be computed on {\em arbitrary matrices}, even if this means losing some confidence in our signal recovery guarantees. Numerical experiments detailed at the end of this work, using simple models for~$e$, seem to suggest that our assumptions on $x^\mathrm{lp}-e$ are not completely unreasonable (cf. Figure~\ref{fig:projerr}). Furthermore, the fact that the true signal $e$ is inherently structured means that, in principle, these statistical fidelity questions would arise with {\em any} model on $e$.

Our contribution here is twofold. First, assuming a Gaussian model or bounded independent model on the nullspace of the matrix $A$ in~(\ref{eq:min-card}), we show that testing if the NSP condition~(\ref{eq:ineq-rnd}) holds with high probability amounts to bounding the value of two classic graph partitioning problems: MaxCut and $k$-Dense-Subgraph. Second, we show new approximation results for semidefinite relaxations of the $k$-Dense-Subgraph problem when the graph weight matrix is positive semidefinite but has coefficients of {\em arbitrary sign}. This result has applications outside of the compressed sensing context discussed in this paper, and is directly related to {\em correlation clustering} for example. Solving a $k$-Dense-Subgraph problem on a (positive semidefinite) correlation matrix (modeling similarities between variables) isolates a highly correlated $k$-cluster of variables. Here, we use these approximation results to show that our weak recovery conditions can be certified in polynomial time for arbitrary matrices even when the target cardinality $k$ is near the true recovery threshold $k^*$ (i.e. a log term away). This allows us to break the $\sqrt{k^*}$ barrier that plagued all tractable conditions for recovery developed so far \citep{dAsp08b,Judi08,dAsp08a}. The current state of the art in checking recovery conditions is that we have conditions that we can trust but cannot fully test. We focus here on conditions we can test, but cannot fully trust.

The paper is organized as follows. Conditions for sparse recovery with high probability, given a model on the nullspace of the sampling matrix are derived in Section~\ref{s:weak}. The performance of these conditions on random matrices and links with the restricted isometry property are discussed in~Section~\ref{s:rip}. Section~\ref{s:bounds} derives semidefinite relaxations and approximation results for the graph partitioning problems used in testing these weak recovery conditions. Section~\ref{s:algos} brielfy discusses the complexity of solving these relaxations. Section~\ref{s:tightness} shows that the relaxations detailed in Section~\ref{s:bounds} allow us to certify weak recovery for near optimal values of the signal cardinality~$k$, thus breaking the $\sqrt{n}$ barrier that plagues tractable (uniform, deterministic) recovery conditions. Finally, we present some numerical experiments in Section~\ref{s:numres}.

\subsection*{Notation}
For $x\in\reals^{n}$, we write $\|x\|_{k,1}$ the sum of the magnitudes of the $k$ largest coefficients of $x$. When $X\in\reals^{m\times n}$, $X_i$ is the $i^{th}$ row of $X$, $\|X\|_2$ the spectral norm and $\|X\|_F$ the Frobenius (Euclidean) norm of $X$. For matrices $A,B\in\reals^{m\times n}$, we write $A \otimes B$ their Kronecker product and $A \circ B$ their Schur (componentwise) product. We write $\NumRank(X)$ the numerical rank of the matrix $X$, with $\NumRank(X)=\|X\|_F^2/\|X\|_2^2$, and $\NumCard(x)$ is the numerical cardinality of a vector $x$, with $\NumCard(x)=\|x\|_1^2/\|x\|_2^2$. Finally, we write $x \preceq_c y$ when $\Expect[f(x)] \leq \Expect[f(y)]$ for any {\em convex} function $f:\reals^n \rightarrow \reals$ (this is usually called convex majorization).

\section{Weak recovery conditions}\label{s:weak}
To highlight the central role of the NSP condition in $\ell_1$ decoding, we begin by adapting a result from \citet[Th. 4.3]{Cohe06} which uses the constant $\alpha_k$ to bound the reconstruction error when decoding the observations $Ae$ by solving problem~\eqref{eq:min-ell1}. Recall that $x^\mathrm{lp}$ is the solution to the linear program in~(\ref{eq:min-ell1}) and~$e$ the true signal.

\begin{proposition}\label{prop:recon-error}
Suppose that $\|x^\mathrm{lp}-e\|_{k,1}\leq \alpha_k \|x^\mathrm{lp}-e\|_1$ for some $\alpha_k<1/2$, where $e\in\reals^n$ and $x^\mathrm{lp}\in\reals^n$ solves problem~(\ref{eq:min-ell1}), then $A(x^\mathrm{lp}-e)=0$ and
\BEQ\label{eq:error-bnd}
\|x^\mathrm{lp}-e\|_1 \leq \frac{2}{(1-2\alpha_k)} ~ \textstyle\min_{\left\{y\in\sreals^n:\,\Card(y)\leq k\right\}} \|y-e\|_1,
\EEQ
where the right-hand side is proportional to the best $\ell_1$ reconstruction error on $e$ using a signal with cardinality $k$.
\end{proposition}
\begin{proof}
We adapt the proof of \citet[Th. 4.3]{Cohe06}. Because $x^\mathrm{lp}$ solves (\ref{eq:min-ell1}), we have $\|x^\mathrm{lp}\|_1\leq \|e\|_1$ since $e$ is feasible. Denoting by $T$ the indices of the $k$ largest coefficients in absolute value of $e$ and by $\eta=x^\mathrm{lp}-e$ the reconstruction error, we write
\[
\|x^\mathrm{lp}_T\|_1+\|x^\mathrm{lp}_{T^c}\|_1 \leq \|e_T\|_1+\|e_{T^c}\|_1
\]
and triangular inequalities yield
\[
\|e_T\|_1 - \|\eta_T\|_1 + \|\eta_{T^c}\|_1 - \|e_{T^c}\|_1 \leq \|e_T\|_1+\|e_{T^c}\|_1
\]
hence
\[
\|\eta_{T^c}\|_1 \leq \|\eta_T\|_1 + 2 \|e_{T^c}\|_1.
\]
Note that by definition of $T$, we have $\|e_{T^c}\|_1=\min_{\left\{y\in\sreals^n:\,\Card(y)\leq k\right\}} \|y-e\|_1$. From our assumption on $\eta$ and by definition of $\|\cdot\|_{k,1}$, $|T|=k$ means
\[
\|\eta_T\|_1\leq \|\eta\|_{k,1} \leq \alpha_k \|\eta\|_1=\alpha_k (\|\eta_T\|_1+\|\eta_{T^c}\|_1)
\]
hence
\[
\|\eta_T\|_1 \leq \frac{\alpha_k}{1-\alpha_k} \|\eta_{T^c}\|_1
\]
which then yields
\[
\|\eta_{T^c}\|_1 \leq \frac{(2-2\alpha_k)}{(1-2\alpha_k)} ~\textstyle \min_{\left\{y\in\sreals^n:\,\Card(y)\leq k\right\}} \|y-e\|_1.
\]
Using the fact that
\[
\|\eta\|_1=\|\eta_T\|_1+\|\eta_{T^c}\|_1\leq \left(1+\frac{\alpha_k}{1-\alpha_k}\right)\|\eta_{T^c}\|_1
\]
we get $\|\eta\|_1 \leq \|\eta_{T^c}\|_1 /(1-\alpha_k)$, which produces the desired result.
\end{proof}

This last result shows that whenever the reconstruction error satisfies~\eqref{eq:ineq-a} with constant $\alpha_k<1$, then the magnitude of this error is at most $2/(1-2 \alpha_k)$ times the best possible reconstruction error achievable using a signal of size $k$.

\subsection{Invariance properties}\label{ss:invar}
Remark that condition~\eqref{eq:ineq-a}, which guarantees recovery of all signals of cardinality less than $k$ can be written
\BEQ\label{eq:ker-nsp}
\|Fy\|_{k,1} \leq \alpha_k \|Fy\|_1, \quad \mbox{for all } y\in\reals^m,
\EEQ
for some $\alpha_k<1/2$, where $F\in\reals^{n \times m}$ is a basis of the nullspace of $A$. This condition is clearly independent of the choice of basis, hence if $F$ satisfies~\eqref{eq:ker-nsp} then so does $FQ$ where $Q$ is any orthogonal matrix.

\subsection{Gaussian model}\label{s:weak-cond-gauss}
In what follows, we will use concentration inequalities to bound both sides of the probabilistic Nullspace Property inequality (\ref{eq:ineq-rnd}), namely check that
\begin{equation*}
\|Fy\|_{k,1} \leq \alpha_k \|Fy\|_1
\end{equation*}
holds with high probability when $y$ is Gaussian with $y\sim{\mathcal N}(0,\idm_{m})$, where $F$ is a basis of the nullspace of $A$. Of course, this means that we implicitly assume that the reconstruction error $x^\mathrm{lp}-e$ follows a Gaussian model. Outside of tractability benefits, there is no fundamental reason to pick a Gaussian distribution on the nullspace of $A$ here, except that its rotational invariance means the basis matrix $F$ only has to be defined up to a rotation. This is consistent with the fact that recovery performance, as characterized by the nullspace property~\eqref{eq:ineq-a}, is only a function of the nullspace of $A$ and not of its matrix representation. Concentration inequalities on Lipschitz functions of Gaussian variables then translate~(\ref{eq:ineq-rnd}) into explicit conditions on the matrix $F$. We begin by the following lemma controlling the left-hand side of this inequality.
\begin{lemma}\label{lem:cond-left}
Suppose $F\in\reals^{n\times m}$ and $y\sim{\mathcal N}(0,\idm_{m})$, then
\[
\Prob\left[\|Fy\|_{k,1}\geq \Expect[\|Fy\|_{k,1}] +x \right]\leq e^{-\frac{x^2}{2\sigma_k^2(F)}}
\]
where
\BEQ\label{eq:sigma}
\sigma_k^2(F) \triangleq \max_{\{u\in\{0,1\}^{2n},\ones^Tu\leq k\}}
u^T
{\scriptsize \left(\BA{rr}
1 & -1\\
-1 & 1\\
\EA\right)}
\otimes FF^T
u.
\EEQ
and
\[
\Expect[\|Fy\|_{k,1}] \leq \sigma_k(F) \sqrt{2 \log (2^k \binom{n}{k})} \leq \sigma_k(F) \sqrt{2k\left(1+\log\left(\frac{2n}{k}\right)\right)}.
\]
\end{lemma}
\begin{proof} We can write the left-hand side of inequality (\ref{eq:ineq-rnd}) as
\[
\|Fy\|_{k,1} = \max_{\{u=(u_+,u_-)\in\{0,1\}^{2n},\ones^Tu\leq k\}} (u_+-u_-)^TFy
\]
which means that $\|Fy\|_{k,1}$ is the maximum of Gaussian variables. Concentration results detailed in \cite[Th. 3.12]{Mass07} for example show that
\[
\Prob\left[\|Fy\|_{k,1}\geq \Expect[\|Fy\|_{k,1}] +x \right]\leq e^{-\frac{x^2}{2\sigma_k^2(F)}}
\]
where $\sigma_k(F)$ is defined as
\[
\sigma_k^2(F)=\max_{\{u=(u_+,u_-)\in\{0,1\}^{2n},\ones^Tu\leq k\}}\Expect\left[\left((u_+-u_-)^TFy\right)^2\right].
\]
We have
\BEAS
\Expect\left[\left((u_+-u_-)^TFy\right)^2\right] &=& \left\|(u_+-u_-)^TF\right\|^2_2 \\
&=& (u_+-u_-)^TFF^T(u_+-u_-)\\
&=&
\left(\BA{c}
u_+\\u_-\EA\right)^T
\left(\BA{rr}
FF^T & -FF^T\\
-FF^T & FF^T\\
\EA\right)
\left(\BA{c}
u_+\\u_-\EA\right),
\EEAS
and we recover (\ref{eq:sigma}) after setting $u=(u_+,u_-)$. Note that we also have
\[
\|Fy\|_{k,1} = \max_{\{v \in {{\mathcal V}_k}\}} v^T Fy,
\]
where ${\mathcal V}_k$ is the set of vectors of size $n$ with exactly $k$ entries equal to $+1$ or $-1$, and $n-k$ zeroes. Each $v^TFy$ is Gaussian with zero mean and variance $v^TFF^Tv$, so $\|Fy\|_{k,1}$ is the maximum of $2^k \binom{n}{k}$ Gaussian random variables. Using \cite[Lem. 2.3]{Mass07} we can therefore bound the expectation as follows
\[
\Expect[\|Fy\|_{k,1}] \leq \sigma_k(F) \sqrt{2 \log (2^k \binom{n}{k})}
\]
and
\[
\binom{n}{k}\leq \frac{n^k}{k!} \leq \left(\frac{ne}{k}\right)^k
\]
yields the desired result.
\end{proof}

Note that the bound in $\exp(-{x^2}/{2\sigma_k^2(F)})$ can be replaced by $2(1-N(x/\sigma_k(F)))$ (see e.g \cite[Thm 3.8]{Mass07}, where $N(x)$ is the Gaussian CDF, which is smaller for larger values of $x$. Expression (\ref{eq:sigma}) means $\sigma_k^2(F)$ is the optimum value of a $k$-Dense-Subgraph problem. Several efficient approximation algorithms have been derived for this graph partitioning problem and will be discussed in Section~\ref{s:bounds}. We now apply similar concentration results to control the fluctuations of the right hand side of inequality (\ref{eq:ineq-rnd}).

\begin{lemma}\label{lem:cond-right}
Suppose $F\in\reals^{n\times m}$ and $y\sim{\mathcal N}(0,\idm_{m})$, then
\[
\Prob\left[\|Fy\|_1 \leq \Expect[\|Fy\|_1] - x \right]\leq e^{-\frac{x^2}{2L^2(F)}}
\]
where
\[
\Expect[\|Fy\|_1]=\sqrt{\frac{2}{\pi}}\sum_{i=1}^n \|F_i\|_2
\]
and $L^2(F)=\max_{v\in\{-1,1\}^n}v^TFF^Tv ~ (=\sigma^2_n(F))$ is bounded by the following MaxCut relaxation
\BEQ\label{eq:L}
\BA{rll}
\frac{2}{\pi} L^2_\mathrm{mxct}(F) \leq L^2(F) \leq L^2_\mathrm{mxct}(F) \triangleq & \mbox{max.} & \Tr(XFF^T)\\
&\mbox{s.t.} & \diag(X)=\ones,X\succeq 0,\\
\EA\EEQ
with, in particular, $L_\mathrm{mxct}(F)\leq\sqrt{n}\|F\|_2$.
\end{lemma}
\begin{proof} We can write
\[
\|Fy\|_1=\max_{v\in\{-1,1\}^n}v^TFy
\]
and \cite[Th. 3.12]{Mass07} shows that
\[
\Prob\left[\|Fy\|_1 \leq \Expect[\|Fy\|_1] - x \right]\leq e^{-\frac{x^2}{2L^2(F)}}.
\]
The fact that $\Expect[|g|]=\sqrt{2/\pi}V$ whenever $g\sim{\mathcal N}(0,V^2)$ produces the expectation, and the Lipschitz constant $L^2(F)$ in this inequality is given by the largest variance
\[
L^2(F)=\max_{v\in\{-1,1\}^n}v^TFF^Tv,
\]
hence is the solution of a graph partitioning problem similar to MaxCut. Relaxation results in \citep{goem95} (in the case where the matrix is nonnegative) and \citep{Nest98a} show that this combinatorial problem can be bounded by solving
\[\BA{ll}
\mbox{maximize} & \Tr(XFF^T)\\
\mbox{subject to} & \diag(X)=\ones,X\succeq 0,\\
\EA\]
which is a semidefinite relaxation in $X\in\symm_n$ of the maximum variance problem (tight up to a factor $\pi/2$). Its dual is written
\[\BA{ll}
\mbox{minimize} & \ones^Tw\\
\mbox{subject to} & FF^T\preceq \diag(w),\\
\EA\]
which is another semidefinite program in the variable $w\in\reals^n$. By weak duality, any feasible point of this last problem gives an upper bound on $L_\mathrm{mxct}(F)$. In particular, the point $w=\lambdamax(FF^T)\ones$ is dual feasible and yields $L_\mathrm{mxct}(F)\leq\sqrt{n}\|F\|_2$.
\end{proof}

The bound detailed in Lemma~\ref{lem:cond-right} is directly related to the $MatrixNorm$ problem discussed in \cite{Nemi01} and \cite{Stei05} or the spin glass models of statistical mechanics. In particular, our approximation bound on $L(F)$ can be directly deduced from the bound on the induced matrix norm~$\|\cdot\|_{2,1}$ derived in \citet[Prop. 1.4]{Stei05}. Note also that the mean $\Expect[\|Fy\|_1]=\sqrt{{2}/{\pi}} \sum_{i=1}^n \|F_i\|_2$ is typically much larger than the factor $L(F)$ controlling concentration. In fact, we can write $\sum_{i=1}^n \|F_i\|_2 = \|F\|_F {\NumCard(\{\|F_i\|_2\})}^{1/2} = \|F\|_2 {\NumRank(F)^{1/2}\NumCard(\{\|F_i\|_2\})^{1/2}}$.
Combining the last two lemmas, we show the following proposition, which is our main recovery condition.

\begin{proposition}\label{prop:weak-cond}
If $F\in\reals^{n\times m}$ satisfies
\BEQ\label{eq:weak-cond}
\left(\sqrt{2k\left(1+\log\frac{2n}{k}\right)}+\beta \right)\sigma_k(F) \leq  \left(\sqrt{\frac{2}{\pi}} \sum_{i=1}^n \|F_i\|_2 - \beta L(F)  \right) \alpha_k
\EEQ
for some $\beta>0$, where $\sigma_k(F)$ was defined in (\ref{eq:sigma}) and $L(F)$ in (\ref{eq:L}), then the sparse recovery condition (\ref{eq:ineq-rnd}) will be satisfied with probability $1-2e^{-\beta^2/2}$ when $y\sim{\mathcal N}(0,\idm_{m})$.
\end{proposition}
\begin{proof}
We combine the bounds of Lemmas \ref{lem:cond-left} and \ref{lem:cond-right}, requiring them to hold with probability $1-e^{-\beta^2/2}$.
\end{proof}

We finish this section by showing that the function $\sigma_k(F)$ defined in (\ref{eq:sigma}) is strictly increasing with $k$ whenever the diagonal of $FF^T$ is positive, which will prove useful in the results that follow.
\begin{lemma}\label{lem:sigma-inc}
Let $F\in\reals^{n\times m}$, the function $\sigma_k(F)$ is strictly increasing in $k\in[1,n]$ whenever the diagonal of $FF^T$ is positive, with
\[
\sigma_1(F)=\max_{i=1,\ldots,n} (FF^T)_{ii} \quad \mbox{and} \quad \sigma_n(F)=L(F)
\]
where $L(F)$ is defined in Lemma \ref{lem:cond-right}.
\end{lemma}
\begin{proof} We can write
\BEAS
\sigma_k^2(F) &= &\max_{\{(u_+,u_-)\in\{0,1\}^{2n},\ones^Tu\leq k\}}\left\|(u_+-u_-)^TF\right\|^2_2\\
&=& \max_{\{v\in\{0,1\}^{n}, \ones^Tv\leq k, u\in\{-1,1\}^{n}\}}u^T(vv^T\circ FF^T)u.
\EEAS
Let us call $v(k),u(k)$ the optimal solutions of the maximization problem with optimal value~$\sigma_k^2(F)$, and let $J=\{i\in[1,n]:v(k)_i\neq 0\}$ be the support of $v(k)$. If we pick $i\in[1,n]$, outside of $J$, we have
\BEAS
\sigma_{k+1}^2(F) & \geq & u(k)^T(v(k)v(k)^T\circ FF^T)u(k) +(FF^T)_{ii} + \max_{u_i\in\{-1,1\}} 2 u_i\left( \sum_{j\in J} u(k)_j (FF^T)_{ij}\right)\\
& = & \sigma_k^2(F) +(FF^T)_{ii} +  2 \left| \sum_{j\in J} u(k)_j(FF^T)_{ij}\right|
\EEAS
Hence the difference between $\sigma_{k+1}^2(F)$ and $\sigma_{k}^2(F)$ is at least $\max_{j\in J} (FF^T)_{jj}$. This means that $\sigma_k(F)$ is increasing and bounded by
\[
\max_{u\in\{-1,1\}^n}u^TFF^Tu,
\]
which is the maximization problem defining $L^2(F)$ in Lemma~\ref{lem:cond-right}.
\end{proof}

\subsection{Independent, bounded model}
\label{s:weak-cond-bnd}
The previous section showed that enforcing condition~(\ref{eq:ineq-rnd}) with high probability for Gaussian vectors $y$ meant controlling the ratio between the Lipschitz constant $\sigma_k(F)$ of the norm $\|Fy\|_{1,k}$ and the norm $\sum_{i=1}^n \|F_i\|_2$. In what follows, we will show that the same quantities control the concentration of $\|Fy\|_{1,k}$ and $\|Fy\|_1$ when the coefficients of $y$ are independent and bounded. Once again, because $F$ is defined up to a rotation here, these results are easily extended to the case where $y=Qu$ with $Q^TQ=\idm$ and the variables $u$ are independent and bounded. We can write a weak recovery condition for this bounded model, similar to condition (\ref{eq:weak-cond}).

\begin{proposition}\label{prop:bounded-ok}
Let $F\in\reals^{n \times m}$ and suppose
\BEQ\label{eq:weak-cond-bnd}
\Expect[\|Fy\|_{1,k}] + \beta \sigma_k(F) \leq (\Expect[\|Fy\|_{1}] -\beta L(F))\alpha_k
\EEQ
for some $\beta>0$, where $\sigma_k(F)$ was defined in (\ref{eq:sigma}) and $L(F)$ in (\ref{eq:L}), then the sparse recovery condition (\ref{eq:ineq-rnd})
\[
\|Fy\|_{k,1} \leq \alpha_k \|Fy\|_1
\]
will be satisfied with probability $1-2ce^{-\beta^2/c\Delta^2}$, where $c>0$ is an absolute constant, when the coefficients of $y\in\reals^m$ are independent and bounded, with $\|y\|_\infty\leq \Delta$.
\end{proposition}
\begin{proof} As pointwise suprema of affine functions in $y$, the functions $\|Fy\|_{1,k}$ and $\|Fy\|_{1}$ are convex and Lipschitz with constants bounded by $\sigma_k(F)$ and $L(F)$ respectively (see the proofs of Lemmas \ref{lem:cond-left} and \ref{lem:cond-right}). If the coefficients of $y\in\reals^m$ are independent and bounded, with $\|y\|_\infty\leq \Delta$, Talagrand's inequality \citep[Corr. 4.10]{Ledo05} then shows that
\[
\Prob\left[|\Expect[\|Fy\|_{1,k}]-\|Fy\|_{1,k}|\geq t  \right] \leq C e^{-\frac{t^2}{c\sigma_k^2(F)\Delta^2}}
\]
and
\[
\Prob\left[|\Expect[\|Fy\|_{1}]-\|Fy\|_{1}|\geq t  \right] \leq C e^{-\frac{t^2}{cL^2(F)\Delta^2}}
\]
where $c$ is an absolute constant, hence the desired result.
\end{proof}

The parallel with the Gaussian case can be made even more explicit using the following simple majorization result.

\begin{lemma} \label{lem:majoriz}
Let $V\subset\reals^n$ be a finite set. Suppose the variables $\{y_i\}_{i=1,\ldots,n}$ are independent with support in $[-1,1]$, then
\[
\Expect[\sup_{v \in V} v^Ty] \leq \sigma \sqrt{\pi \log |V|}
\]
where $\sigma=\max_{v\in V} \|v\|_2$.
\end{lemma}
\begin{proof}
If the variables $y_i$ are independent, supported in $[-1,1]$, then $y \preceq_c g$ where $g\sim {\mathcal N}(0,\frac{\pi}{2} \idm_n) $ is a Gaussian vector \citep[Prop.\,10.3.2]{Ben09}. The supremum $\sup_{v \in V} v^Ty$ is a pointwise maximum of affine functions of $y$, hence is convex in $y$, so $y \preceq_c g$ implies $\Expect[\sup_{v \in V} v^Ty]\leq \Expect[\sup_{v \in V} v^Tg]$. Finally, \citep[Th.\,3.12]{Mass07} shows that $\Expect[\sup_{v \in V} v^Tg]\leq \sigma \sqrt{\pi \log |V|} $.
\end{proof}

If we take $V$ in Lemma~\ref{lem:majoriz} to be the set of vectors of size $n$ with exactly $k$ entries equal to $+1$ or $-1$, and $n-k$ zeroes, this result shows that, when the coefficients of $y$ are supported on $[-1,1]$ and independent, then $\Expect[\|Fy\|_{1,k}]$ is bounded by $\frac{\pi}{2}\Expect[\|Fg\|_{1,k}]$ with $g$ Gaussian. Alternatively, both expectations  in (\ref{eq:weak-cond-bnd}) can be evaluated efficiently. In fact Hoeffding's inequality shows that if we need to estimate these quantities with precision $\epsilon$ and confidence $1-\beta$, we need at least $N$ samples of either $\|Fy\|_{1,k}$ or $\|Fy\|_1$, with
\[
N= \frac{D^2 \log(2/\beta)}{2 \epsilon^2}
\]
where $D=\max_{\|y\|_\infty\leq \Delta} \|Fy\|_1$ is an upper bound on both norms whenever $\|y\|_\infty\leq \Delta$.

\section{Weak recovery and restricted isometry} \label{s:rip}

In this section, we show that some random matrices satisfy our weak recovery condition~\eqref{eq:weak-cond} for near optimal values of the cardinality $k$ (i.e. in scenarios where the number $m$ of linear samples required to recover a signal is a small multiple of the number $k$ of nonzero components in that signal). We show in particular that in some cases, matrices satisfying the restricted isometry property defined in \citet{Cand05} also satisfy condition~\eqref{eq:weak-cond}. Here however, the restricted isometry is tested on the nullspace basis~$F$ instead of the coding matrix~$A$, so the compressed sensing interpretation is lost, but this connection allows us to {\em recycle all known results on restricted isometry thresholds for random matrices} and easily derive weak recovery thresholds from condition~\eqref{eq:weak-cond}. In the next section, we will see that the most important difference between RIP and the weak condition detailed here is that~\eqref{eq:weak-cond} can be tested efficiently while RIP is intractable. Here, we simply check that our weak condition~\eqref{eq:weak-cond} is indeed satisfied by good coding matrices.

We first show that the $k$-Dense-Subgraph problem computing $\sigma_k(F)$ in~\eqref{eq:sigma} is inherently simpler than the sparse eigenvalue problem used in testing the restricted isometry property. We then show that matrices $F$ such that $F^T$ satisfies the restricted isometry property defined in \citet{Cand05} at a near-optimal cardinality $k$, also satisfy our weak recovery condition~(\ref{eq:weak-cond}) for similar values of $k$. This allows us to recycle all known results on the RIP for random matrices and show in particular that Gaussian matrices satisfy condition~\eqref{eq:weak-cond} at near optimal values of $k$. 

Roughly speaking, our main objective here is to show that for good coding matrices $\sigma_k(F)$ grows as $\sqrt{k}$ while $L(F)$ is of order $\sqrt{n}$ and $\sum_{i=1}^n \|F_i\|_2$ is of order $n$ (up to a normalizing factor in condition~\eqref{eq:weak-cond}). For completeness, we have also included a direct proof of these facts in the appendix, using standard concentration arguments instead of RIP.

\subsection{Sparse eigenvalues and $k$-Dense-Subgraph.} We will see in the next section that approximating the $k$-Dense-Subgraph problem is significantly easier than testing RIP or the nullspace property. There is in fact a direct connection between the sparse eigenvalue and $k$-Dense-Subgraph problems. The $k$-Dense-Subgraph problem used in bounding $\sigma_k(F)$ is written
\[
\sigma^2_k(F)=\max_{\substack{u\in\{0,1\}^{2n}\\\ones^Tu\leq k}} ~ u^TMu \quad\mbox{where}\quad
M=
{\scriptsize \left(\BA{rr}
1 & -1\\
-1 & 1\\
\EA\right)}
\otimes FF^T
\]
in the variable $u\in\{0,1\}^{2n}$. On the other hand, the problem of computing a sparse maximum eigenvalue to check the restricted isometry property can be written
\[
\lambda^k_{\mathrm{max}}(FF^T) = \max_{\substack{u\in\{0,1\}^n\\\ones^Tu\leq k}} ~\max_{\|x\|=1} ~ u^T(FF^T\circ xx^T)u
\]
in the variables $x\in\reals^n$, $u\in\{0,1\}^n$. We observe that computing sparse eigenvalues (hence test RIP) means solving a $k$-Dense-Subgraph problem over the result of an inner eigenvalue problem in $x$, while bounding~$\sigma_k(F)$ only requires solving a $k$-Dense-Subgraph problem over a fixed matrix $M$, hence is significantly easier.

\subsection{Weak NSP and random matrices}

Following \citet{Cand05}, we will say that a matrix $A\in\reals^{m\times n}$ satisfies the {\em restricted isometry property} (RIP) at cardinality $k>0$ if there is a constant $\delta_k>0$ such that
\[
\|x\|^2_2 (1-\delta_k) \leq \|Ax\|^2_2 \leq (1+\delta_k) \|x\|^2_2
\]
for all sparse vectors $x\in\reals^n$ such that $\Card(x)\leq k$. We now show that the RIP allows us to closely control the values of $\sigma_k(F)$ and $L(F)$. This will allow us to directly recycle all known results on RIP for random matrices and apply them to the weak recovery condition considered here. We start by a technical lemma bounding the values of $\sigma_k(F)$ and $\|F_i\|_2$ for RIP matrices.

\begin{lemma}\label{lem:rip-sigma}
Suppose the matrix $F^T\in\reals^{m\times n}$ satisfies the restricted isometry property with constant $\delta_k>0$ at cardinality $k$, then
\BEQ\label{eq:weak-rip}
\sigma_k(F)\leq \sqrt{k(1+\delta_k)} \quad \mbox{and} \quad \|F_i\|_2 \geq \sqrt{1-\delta_1}.
\EEQ
and $(k/n)^2 L^2(F)\leq \sigma^2_k(F)$.
\end{lemma}
\begin{proof}
We get
\BEAS
\sigma_k^2(F) & = & \max_{\{(u_+,u_-)\in\{0,1\}^{2n},\ones^Tu\leq k\}}\left\|(u_+-u_-)^TF\right\|^2_2\\
& = & \max_{\{(u_+,u_-)\in\{0,1\}^{2n},\ones^Tu\leq k\}}(u_+-u_-)^TFF^T(u_+-u_-)\\
& \leq & (1+\delta_k)~\max_{\{(u_+,u_-)\in\{0,1\}^{2n},\ones^Tu\leq k\}} \|u_+-u_-\|_2^2\\
& \leq & (1+\delta_k) k
\EEAS
because $F^T$ satisfies the RIP and $\Card(u_+-u_-)\leq k$. Plugging Euclidean basis vectors in the RIP also means $(1-\delta_1) \leq \|F_i\|^2_2$ for $i=1,\ldots,n$. Lemma \ref{lem:sigma-inc} showed that $L(F)=\sigma_n(F)$ and combining this with the lower bound in \citep[Lem.1]{Sriv98} on the performance of the greedy algorithm in \S\ref{sss:greedy} shows that $(k/n)^2 L^2(F)\leq \sigma^2_k(F)$.
\end{proof}
 We now use this last lemma to show the main result of this section, which proves that if a matrix $F$ satisfies RIP then $F^T$ will satisfy the weak recovery condition~(\ref{eq:weak-cond}) in the optimal regime where $k$ is proportional to~$n$. In other words, this result shows that our weak recovery condition is satisfied by optimal matrices, hence is indeed weaker than existing recovery conditions.

\begin{proposition}\label{prop:weak-rip}
Suppose $F^T\in\reals^{m\times n}$ satisfies the restricted isometry property with constant $\delta_k$ with $0<\delta_k<c<1$ at cardinality $k$, where $c$ is an absolute constant. Suppose that $k\leq n$, $k\rightarrow \infty$ as $n\rightarrow \infty$ and $\limsup_{n\rightarrow \infty} k/n=\kappa$. Then $F$ satisfies condition~\eqref{eq:weak-cond} for $n$ large enough with $\alpha_k<1/2$, provided that $f_c(\kappa)<1/2$, where $f_c$ is defined as 
$$
f_c(x)=x\frac{\sqrt{\pi(1+c)}}{\sqrt{1-c}}\sqrt{\left(1+\log\frac{2}{x}\right)}\;, 
$$
and $f_c(0)=0$\;.
\end{proposition}
\begin{proof} When $F^T$ satisfies the RIP, Lemma~\ref{lem:rip-sigma}  \,above shows
\[
\sigma_k(F) \leq \sqrt{k(1+\delta_k)}
\]
and, using  $L(F) \leq (n/k) \sigma_k(F)$ (see Lemma~\ref{lem:rip-sigma}), we then get $L(F)\leq n k^{-1/2}\sqrt{(1+\delta_k)}$. Therefore, 
\[
\sqrt{\frac{2}{\pi}} \sum_{i=1}^n \|F_i\|_2 - \beta L(F)   \geq n\sqrt{\frac{2(1-\delta_1)}{\pi}}   - \beta  n \sqrt{(1+\delta_k)/k}
\]
for any $\beta>0$. We also note that $\delta_1\leq \delta_k<c$ so that
\begin{align*}
\sqrt{\frac{2}{\pi}} \sum_{i=1}^n \|F_i\|_2 - \beta L(F)   &\geq n\sqrt{\frac{2(1-\delta_k)}{\pi}}   - \beta  n \sqrt{(1+\delta_k)/k}\\
&>
n\left(\sqrt{\frac{2(1-c)}{\pi}}   - \beta  \sqrt{(1+c)/k}\right).
\end{align*}
Using the fact that $\sigma_k(F) \leq \sqrt{k(1+c)}$, it is clear that if 
$$
\alpha_k\left[\sqrt{\frac{2}{\pi}} \sum_{i=1}^n \|F_i\|_2 - \beta L(F)\right]\geq \sqrt{k(1+c)} \left[\sqrt{2k\left(1+\log\frac{2n}{k}\right)}+\beta\right],
$$
then Equation~\eqref{eq:weak-cond} holds. Therefore, if 
$$
\alpha_k \, n\left(\sqrt{\frac{2(1-c)}{\pi}}   - \beta  \sqrt{(1+c)/k}\right)\geq \sqrt{k(1+c)} \left[\sqrt{2k\left(1+\log\frac{2n}{k}\right)}+\beta\right],
$$
or equivalently, assuming $k>\pi\beta^2(1+c)/2(1-c)$, if
$$
\alpha_k\geq \frac{k}{n}\frac{\sqrt{1+c}}{\sqrt{\frac{2(1-c)}{\pi}}-\beta \sqrt{\frac{1+c}{k}}}\left[\sqrt{2\left(1+\log \frac{2n}{k}\right)}+\frac{\beta}{\sqrt{k}}\right]\triangleq \Gamma(k,n,c,\beta),
$$
then Equation~\eqref{eq:weak-cond} holds. It is therefore clear that if $\Gamma(k,n,c,\beta)<1/2$, we can find $\alpha_k<1/2$ such that Equation~\eqref{eq:weak-cond} holds. Notice that as $k\rightarrow \infty$, we have 
$$
\Gamma(k,n,c,\beta)\sim \frac{k}{n}\frac{\sqrt{\pi(1+c)}}{\sqrt{1-c}}\sqrt{1+\log \frac{2n}{k}}=f_c(k/n).
$$
Elementary analysis shows that $f_c$ is a continuous increasing function on $[0,1]$. 

Recall now, that by assumption, $k\rightarrow \infty$ as $n\rightarrow \infty$ and $\limsup_{n\rightarrow \infty}\frac{k}{n}=\kappa\geq 0$ with $\kappa$  such that $f_c(\kappa)<1/2$. We therefore conclude - by considering $\limsup_{n\rightarrow \infty} \Gamma(k,n,c,\beta)$ - that when $n$ is large enough, Equation~\eqref{eq:weak-cond} holds with $\alpha_k<1/2$ under our assumptions.
\end{proof}

This last result shows that $F$ satisfies the weak recovery condition in~(\ref{eq:weak-cond}) at cardinalities near~$k$ when~$F^T$ satisfies the RIP at cardinality $k$, in the optimal regime where $k$ is proportional to $n$. 

\section{Bounds on $L(F)$ and $\sigma_k(F)$ using graph partitioning relaxations} \label{s:bounds}
In Section~\ref{s:weak-cond-gauss}, we showed that if the matrix $F\in\reals^{n \times m}$ satisfied the weak recovery condition~(\ref{eq:weak-cond}), which read
\[
\left(\sqrt{2k\log\left(1+\frac{2n}{k}\right)}+\beta \right)\sigma_k(F) \leq  \left(\sqrt{\frac{2}{\pi}} \sum_{i=1}^n \|F_i\|_2 - \beta L(F)  \right) \alpha_k,
\]
for some $\beta>0$, then the recovery condition in (\ref{eq:ineq-rnd}) would be satisfied with probability $1-2e^{-\beta^2/2}$ when $y$ is Gaussian. Testing this weak recovery condition essentially hinged on bounding the Lipschitz constants $\sigma_k(F)$ and $L(F)$. In Section~\ref{s:weak-cond-bnd} we showed that the same quantities allowed us to check the weak recovery condition in a more general model where $y$ is bounded. As we will see below, efficient approximation results on these graph partitioning problems produce relatively tight bounds on both $\sigma_k(F)$ and $L(F)$. In particular, these bounds are tight enough to allow condition~(\ref{eq:ineq-rnd}) to be tested in polynomial time at near-optimal values of the cardinality $k$.

\subsection{Bounding $L(F)$: MaxCut} \label{ss:L-bound}
We have observed in Lemma \ref{lem:cond-right} that the constant $L(F)$ on the right hand side of condition (\ref{eq:weak-cond}) is defined as
\BEQ\label{eq:knp}
L^2(F)=\max_{v\in\{-1,1\}^n}v^TFF^Tv.
\EEQ
This is an instance of a graph partitioning problem similar to MaxCut. \citet{goem95} (when the matrix is nonnegative) and \citet{Nest98a} show that the following relaxation
\BEQ\label{eq:maxcut}
\BA{rll}
L^2(F) \leq L_\mathrm{mxct}^2(F) = & \mbox{max.} & \Tr(XFF^T)\\
&\mbox{s.t.} & \diag(X)=\ones,X\succeq 0,\\
\EA\EEQ
which is a (convex) semidefinite program in the variable $X\in\symm_n$, is tight up to a factor $\pi/2$. This means that $\sqrt{2/\pi}L_\mathrm{mxct}(F)\leq L(F) \leq L_\mathrm{mxct}(F)$. The dual of this last program is written
\[\BA{ll}
\mbox{minimize} & \ones^Tw\\
\mbox{subject to} & FF^T\preceq \diag(w),\\
\EA\]
which is another semidefinite program in the variable $w\in\reals^n$. By weak duality, any feasible point of this last problem gives an upper bound on $L(F)$.

\subsection{Bounding $\sigma_k(F)$: k-Dense-Subgraph} \label{ss:sigma-bound}
On the left hand side of (\ref{eq:weak-cond}), the constant $\sigma_k^2(F)$ is computed as
\BEQ\label{eq:q-knap}
\BA{rll}
\sigma_k^2(F) =& \mbox{max.} & u^TMu\\
&\mbox{s.t.} &\ones^Tu\leq k\\
&& u\in\{0,1\}^{2n},
\EA\EEQ
in the binary variable $u$, where $M\in\symm_{2n}$ is positive semidefinite, with
\BEQ\label{eq:m}
M=
{\scriptsize \left(\BA{rr}
1 & -1\\
-1 & 1\\
\EA\right)}
\otimes FF^T,
\EEQ
here. This is a graph partitioning problem known as k-Dense-Subgraph, which seeks to find a subgraph~$S$ of the graph of $M$, with at most $k$ nodes and maximum edge weight $\sum_{(i,j)\in S} M_{ij}$, see \citet{Kort93,Aror95,Feig01a,Feig01,Han02,Bill06} among others for details. Note that in our application here, $M$ is typically dense and its coefficients can take negative values while most of the references cited above consider graphs with nonnegative (often sparse) weight matrices. The $k$-DenseSubgraph problem can also be seen as an instance of the Quadratic Knapsack problem (see \citet{Lin98,Pisi07} for a general overview). We will see that elementary greedy or random sampling algorithms already produce satisfactory approximations. However, their crudeness means that they are outperformed in practice by linear programming or semidefinite relaxation bounds, and we begin by outlining a few of these relaxations below.

\subsubsection{A Greedy Algorithm.} \label{sss:greedy}
We now recall the greedy elimination procedure described by e.g. \citet{Sriv98}, which extracts a $k$-subgraph out of a larger graph containing the optimal solution. Suppose we are given a weight matrix $M\in\symm_n$, and assume we know an index set $I\in[1,n]$ such that the weight $w(I)=\sum_{i,j\in I} M_{ij}$ of the subgraph with vertices in $I$ is an upper bound on $\sigma_k^2(F)$ of the $k$-Dense-Subgraph problem in (\ref{eq:q-knap}). If $|I|\leq k$, then $I$ is optimal, otherwise we can greedily prune $|I|-k$ vertices from the graph and \citet[Lem.1]{Sriv98} show that the pruned subgraph must have weight at least
\[
\frac{k(k-1)}{|I|(|I|-1)}w(I).
\]
When the weight matrix $M$ is nonnegative, the full graph weight $w([1,n])$ produces an obvious upper bound on $w(I^*)$. The situation is slightly more complex when $M$ has negative coefficients, as in the particular instance considered here in (\ref{eq:sigma}). In Proposition \ref{prop:test}, we show how to produce an upper bound $w(I)$ by solving the MaxCut relaxation (\ref{eq:maxcut}).

\subsubsection{Semidefinite Relaxation.}\label{subsubsec:SemiDefRelax} Many different relaxations have been developed for the $k$-Dense-Subgraph and Quadratic Knapsack problem and we highlight some of them in what follows. Semidefinite relaxations were derived in \cite{Helm00a} to bound $\sigma_k^2(F)$. In particular, the SQK2 relaxation in \cite{Helm00a} yields
\BEQ\label{eq:sqk2}
\BA{rll}
\sigma_k^2(F) \leq & \mbox{max.} & \Tr(MX)\\
& \mbox{s.t.} & \ones^TX\ones \leq k^2\\
& & X - \diag^2(X) \succeq 0,
\EA\EEQ
which is a semidefinite program in the variable $X\in\symm_n$. Note that the constraint $X - \diag^2(X)$ is a Schur complement, hence is convex in $X$. Adaptively adding further constraints as in \cite{Helm00a} can further tighten this relaxation. In particular, adding constraints of the type
\BEQ\label{eq:addc}
\sum_{j=1}^n X_{ij} \leq kX_{ii}
\quad \mbox{or} \quad
\sum_{j=1}^n(X_{jj}-X_{ij}) \leq (1-X_{ii})
\EEQ
for some $i=1,\ldots,n$, sometimes significantly improves tightness. Another simple relaxation formulated in \cite{Helm00a} bounds (\ref{eq:knp}) when $k\geq 2$ by solving
\BEQ\label{eq:sqk3}
\BA{rll}
\sigma_k^2(F) \leq & \mbox{max.} & \Tr(MX)\\
& \mbox{s.t.} & \Tr((\ones\ones^T-\idm)X) \leq k(k-1)\\
& & X - \diag^2(X) \succeq 0,
\EA\EEQ
in the variable $X\in\symm_n$. This last relaxation is tighter than (\ref{eq:sqk2}) but not as tight as its refinements using the additional constraints in (\ref{eq:addc}). Another relaxation detailed in \cite{Feig01} first writes (\ref{eq:q-knap}) as a binary optimization problem over $\{-1,1\}^n$, then bounds it by solving
\BEQ\label{eq:sdp-ksub}
\BA{ll}
\mbox{maximize} & \Tr(M(\ones\ones^T+y\ones^T+\ones y^T+Y))\\
\mbox{subject to} & Y\ones=y(2k-n)\\
& \diag(Y)=\ones,Y\succeq 0,
\EA\EEQ
which is a semidefinite program in the variable $Y\in\symm_n$. We refer the reader to \cite{Helm00a} for details on the tightness and complexity of these various semidefinite relaxations.

Fortunately, even though the $k$-Dense-Subgraph problem is NP-Hard, simple randomized or greedy algorithms reach good approximation ratios (\citet{Aror95} even produced a PTAS in the dense nonnegative case). While many tightness results have been derived on the semidefinite relaxations detailed above (see e.g. \citet{Han02a}), most of them producing approximation ratios of $k/n$ or better, existing results do not apply when {\em the coefficients of $M$ have arbitrary signs}.  Here, we show a similar approximation ratio when the graph weight matrix $M$ is allowed to have some negative coefficients but is positive semidefinite.

\begin{proposition}\label{prop:approx-sdp}
Suppose $M\in\symm_n$ is positive semidefinite. Define
\[
{\mathcal D}_k(M)=\max_{\substack{u\in\{0,1\}^{n}\\\ones^Tu\leq k}} ~ u^TMu,
\]
the relaxation
\BEQ \label{eq:sdp-relax}
\BA{rll}
SDP_k(M)=& \mbox{max.} & \Tr MX\\
& \mbox{s.t.} & 0 \leq X_{ij} \leq 1\\
& & \Tr X=k,\, X \succeq 0,
\EA\EEQ
satisfies, for $n$ large enough and $k\geq n^{1/3}$,
\[
\frac{k}{n} \mu(n,k) \left(\frac{1}{4}\Tr MG +\frac{1}{2\pi} SDP_k(M) \right) \leq {\mathcal D}_k(M) \leq SDP_k(M),
\]
where
\[
\mu(n,k)=\left(1-\frac{2}{k^{1/3}} \right) \left(\frac{1}{1-\frac{2\pi n^2}{k^2}e^{-\frac{n^{1/9}}{3}}}\right)\xrightarrow[n \rightarrow \infty]{}1
\]
and $G_{ij}=\sqrt{X_{ii}X_{jj}}$, $i,j=1,\ldots,n$, so in particular $\Tr MG \geq 0$.
\end{proposition}
\begin{proof}
We use a hybrid randomization procedure, mixing the sparse sampling strategy in \cite{Feig97} with the correlation argument in \citet{Nest98a}. Let $X$ be an optimal solution to problem~\eqref{eq:sdp-relax}, w.l.o.g. we can assume $|X_{ii}|>0$, and we define the corresponding (positive semidefinite) correlation matrix $C_{ij}=X_{ij}/\sqrt{X_{ii}X_{jj}},\, i,j=1,\ldots n$ and sample vectors $z\sim{\mathcal N}(0,C)$. For each sample $z$, we define
\[
y_i=
\left\{\BA{l}
1 \quad\mbox{if } z_i \geq 0,\\
0 \quad\mbox{otherwise.}
\EA\right.
\]
As in \cite{Feig97}, we also sample independent variables $u\in\reals^n$ such that
\[
u_i=
\left\{\BA{l}
1 \quad\mbox{with probability } q_i=k\sqrt{X_{ii}}/S,\\
0 \quad\mbox{otherwise.}
\EA\right.
\]
where $S=\sum_{i=1}^n \sqrt{X_{ii}}$. Note that $0\leq q_i\leq 1$ because $0\leq X_{ii} \leq 1$ and $\sum_i X_{ii}=k$. For each sample, we then define $w\in\{0,1\}^n$, with $w_i=u_iy_i,\, i=1,\ldots,n$, so when $i \neq j$
\BEAS
\Expect[w_i w_j] &=& \Prob[z_i \geq 0, z_j \geq 0, u_i=u_j=1]\\
&=& \Prob[z_i\geq 0, z_j \geq 0] \Prob[u_i=1] \Prob[u_j=1]\\
&=& \left(\frac{1}{4} +\frac{1}{2\pi} \arcsin(C_{ij}) \right)\frac{k^2\sqrt{X_{ii}X_{jj}}}{S^2}
\EEAS
and $\Expect[w_i^2] \geq \Prob[z_i\geq 0] \Prob[u_i=1]^2$. If we define $G\in\symm_n$ with $G_{ij}=\sqrt{X_{ii}X_{jj}}$, we conclude that
\[
\Expect[ww^T] \succeq \frac{k^2}{S^2} \left[\frac{1}{4}G +\frac{1}{2\pi}\arcsin (C)\circ G\right].
\]
Because $X,M\succeq 0$ with $\Tr X=k$, we have $S\leq \sqrt{kn}$, and we thus obtain
\BEAS
\Expect[w^TMw] & \geq & \frac{k^2}{S^2} \left(\frac{1}{4}\Tr MG +\frac{1}{2\pi} \Tr (M (\arcsin (C) \circ G)) \right)\\
& \geq & \frac{k}{n} \left(\frac{1}{4}\Tr MG +\frac{1}{2\pi} SDP_k(M) \right)
\EEAS
because $\arcsin(C) \succeq C$ \citep[Corr.\,3.2]{Nest98b}, $\Tr (M (\arcsin (C) \circ G))= \Tr (\arcsin( C )(M \circ G)) $, $C\circ G=X$ and $M,C,G \succeq 0$ so $M \circ G \succeq 0$. Now, let us call $b=\Prob[w^TMw \leq \Expect[w^TMw]/\beta]$ for some $\beta \geq 1$. By construction, because $w^TMw \leq SDP_n(M)$ whenever $w\in\{0,1\}^n$ and
\[
w^TMw \leq \frac{\Expect[w^TMw]}{\beta}~\ones_{\{w^TMw\leq \Expect[w^TMw]/\beta\}}+ SDP_n(M) ~\ones_{\{w^TMw> \Expect[w^TMw]/\beta\}}
\]
we have
\[
\Expect[w^TMw] \leq b \Expect[w^TMw]/\beta + (1-b) SDP_n(M)
\]
so
\[
b \leq 1 - \frac{\beta -1}{\beta SDP_n(M) / \Expect[w^TMw] -1}.
\]
Now, let us call $Y\in\symm_n$ a solution to $SDP_n(M)$; then $k Y/n$ is a feasible point of~\eqref{eq:sdp-relax}, so $SDP_n(M)=\Tr MY\leq \frac{n}{k} \Tr MX$ and the previous paragraph shows
\[
\frac{SDP_n(M)}{\Expect[w^TMw]}= \frac{\Tr MY}{\Expect[w^TMw]} \leq \frac{2 \pi n \Tr MY}{k \Tr MX} \leq \frac{2 \pi n^2}{k^2}\;.
\]
Therefore, for $n$ large enough, setting
\[
\beta \geq \frac{1}{1-\frac{2\pi n^2}{k^2}e^{-k^{1/3}/3}},
\]
ensures
\[
\beta > \frac{1-e^{-k^{1/3}/3}}{1-\frac{SDP_n(M)}{\Expect[w^TMw]}e^{-k^{1/3}/3}}\;.
\]
When the denominator is positive, the previous inequality implies that 
\[
\frac{\beta -1}{\beta SDP_n(M) / \Expect[w^TMw] -1} > e^{-k^{1/3}/3}\;.
\]
Hence, choosing again $n$ large enough to make the denominator positive, we finally have
\[
1-b \geq \frac{\beta -1}{\beta SDP_n(M) / \Expect[w^TMw] -1} > e^{-k^{1/3}/3},
\]
Now, using Chernoff's inequality as in \cite[Lem. 4.1]{Feig97} produces
\[
\Prob\left[\Card(u) - \ones^Tq \geq t \ones^Tq \right]\leq e^{-\frac{t^2 \ones^Tq}{3}},
\]
where $q_i=\Prob[u_i=1]$. We note that here $\ones^Tq=k$ and as in \cite[Th. 4.1]{Feig97}, when $k \geq n^{1/3}$
\[
\Prob\left[\Card(u) \geq k\left(1+k^{-1/3}\right)\right] \leq e^{-k^{1/3}/3}.
\]
This last result, together with the bound on $b$ derived above, shows that
\[
\Prob[w^TMw \geq \Expect[w^TMw]/\beta] = 1-b > e^{-k^{1/3}/3} \geq \Prob\left[\Card(w) \geq k\left(1+k^{-1/3}\right)\right]\;.
\]
Therefore, by sampling enough points~$w$, we can generate a vector $w_0\in\{0,1\}^n$ such that
\[
w_0^TMw_0 \geq \frac{k}{\beta n} \left(\frac{1}{4}\Tr MG +\frac{1}{2\pi} SDP_k(M) \right)
\quad\mbox{and}\quad
\Card(w_0)\leq k\left(1+k^{-1/3}\right)
\]
If we remove no more than $k^{2/3}$ variables from $w_0$ using the backward greedy algorithm described in \citet[Lem.1]{Sriv98} we loose at most a factor
\[
\frac{k(k-1)}{(k+k^{2/3})(k+k^{2/3}-1)}=1-\frac{2}{k^{1/3}}+o\left(\frac{1}{k^{1/3}}\right)
\]
and, from $w_0$, we obtain a point $w_k$ such that
\[
w_k^TAw_k \geq \frac{k}{\beta n} \left(1-\frac{2}{k^{1/3}} \right)\left(\frac{1}{4}\Tr MG +\frac{1}{2\pi} SDP_k(M) \right)
\quad\mbox{and}\quad
\Card(w_k)\leq k,
\]
when $n$ is large enough, which yields the desired result.
\end{proof}

Note that, in the previous result, the condition $k \geq n^{1/3}$ can be replaced by any constraint of the type $k\geq n^\alpha$ where $0<\alpha<1$ with $n^{1/9}$ replaced by $n^{\alpha/3}$.

\section{Complexity} \label{s:algos}
Bounding $L(F)$ and $\sigma_k(F)$ using semidefinite relaxations means solving two maximum eigenvalue minimization problems. Problem~\refp{eq:maxcut}, used for bounding $L(F)$, can be rewritten
\BEQ\label{eq:minmax-maxcut}
\min_{w\in\small{\sreals^n}}~ n\lambdamax(FF^T-\diag(w))-\ones^Tw
\EEQ
while problem~\refp{eq:sqk2} bounding $\sigma_k(F)$ can be written
\BEQ\label{eq:minmax-sigmak}
\min_{\{w,z\in\sreals^n,~y\in\small{\sreals^n}\}} ~(k+1)\lambdamax\left(\bar F + w \bar H + z \bar G + \sum_{i=1}^n y_i \bar E_i \right)-wk(k-1)-z
\EEQ
where
\[
\bar F=\left(\BA{cc}
FF^T & 0\\
0 & 0
\EA\right)
,~
\bar H=\left(\BA{cc}
\ones\ones^T-\idm & 0\\
0 & 0
\EA\right)
,~
\bar G=\left(\BA{cc}
0 & 0\\
0 & 1
\EA\right)
~ \mbox{and} ~
\bar E_i=\left(\BA{cc}
e_ie_i^T & -e_i/2\\
-e_i^T/2 & 0
\EA\right)
\]
where $e_i\in\reals^n$ is the $i^{th}$ Euclidean basis vector. Given a priori bounds on the norm of the solutions, \cite{Nest04a} showed that solving problems~\refp{eq:minmax-sigmak} and~\refp{eq:minmax-maxcut} up to a target precision $\epsilon$ using first-order methods has total complexity growing as
\[
O\left(\frac{n^3\sqrt{\log n}}{\epsilon}\right) \quad\mbox{and}\quad O\left(\frac{n^{3.5}\sqrt{\log n}}{\epsilon}\right)
\]
for problems~\refp{eq:minmax-maxcut} and~\refp{eq:minmax-sigmak} respectively.

\section{Tightness} \label{s:tightness}
We use the convex relaxation result of Proposition~\ref{prop:approx-sdp} to show that if a matrix $F$ satisfies the weak recovery condition~(\ref{eq:weak-cond}) up to cardinality $k^*$, the semidefinite relaxation in~\eqref{eq:sdp-relax} will allow us to certify that $F$ satisfies~(\ref{eq:weak-cond}) at cardinalities very near $k^*$.

\begin{proposition}\label{prop:test}
Suppose the matrix $F\in\reals^{n \times m}$ satisfies the weak recovery condition (\ref{eq:weak-cond}) up to cardinality $k^*=\gamma(n)n$ for some $\gamma(n)\in(0,1)$, $\beta>0$ and $\alpha_{k^*}\in[0,1]$, i.e.
\[
\left(\sqrt{2k^*\log\frac{2n}{k^*}}+\beta \right)\sigma_{k^*}(F) \leq  \left(\sqrt{\frac{2}{\pi}} \sum_{i=1}^n \|F_i\|_2 - \beta L(F)  \right) \alpha_{k^*},
\]
and let $SDP_k(\cdot)$ be defined as in~(\ref{eq:sdp-relax}), we have
\BEQ\label{eq:eff-cond}
\left(\sqrt{2k\log\frac{2n}{k}}+\beta \right) (SDP_{k}(M))^{1/2} \leq  \left(\sqrt{\frac{2}{\pi}} \sum_{i=1}^n \|F_i\|_2 - \beta L(F)  \right) \alpha_{k^*},
\EEQ
for $n$ sufficiently large, when $k\leq \gamma(n)(\log n)^{-1} k^*$, with $M$ defined as in~\eqref{eq:m}.
\end{proposition}
\begin{proof}
Applying the result of Proposition~\ref{prop:approx-sdp} at cardinality $k^*$ shows
\[
(SDP_{k^*}(M))^{1/2} \leq \sigma_{k^*}(F) \sqrt{\frac{2\pi n}{k^*}} \left(1+\frac{o(1)}{{k^*}^{1/3}} \right)^{1/2} \;.
\]
Using $SDP_{k}(M)\leq SDP_{k^*}(M)$, with $k\leq \gamma(n)(\log n)^{-1} k^*$ showing
\[
\frac{\left(\sqrt{2k\log\frac{2n}{k}}+\beta \right)}{\left(\sqrt{2k^*\log\frac{2n}{k^*}}+\beta \right)}\sqrt{\frac{2\pi n}{k^*}} \left(1+\frac{o(1)}{{k^*}^{1/3}} \right)^{1/2}  = o(1)
\]
when $n\rightarrow\infty$, yields the desired result.
\end{proof}

\section{Numerical Results}\label{s:numres}

We start by studying the distribution of the residual error $x^\mathrm{lp}-e$ when $e$ is a random sparse signal. We sample a thousand vectors $e\in\reals^{100}$ with 15 nonzero i.i.d. uniform coefficients. Our (fixed) design matrix $A\in\reals^{m\times n}$ is Gaussian or Bernoulli with $m=30$. We produce a vector of observations $Ae$ and solve the $\ell_1$ reconstruction problem in~\refp{eq:min-ell1} and record the value of $x^\mathrm{lp}-e$ projected along a fixed (randomly chosen) direction $v$. The histogram of these values is plotted in Figure~\ref{fig:projerr}.

\begin{figure}[ht]
\begin{center}
\begin{tabular}{cc}
\psfrag{errpoj}[t][b]{Projected Error (Gaussian)}
\includegraphics[width=.45\textwidth]{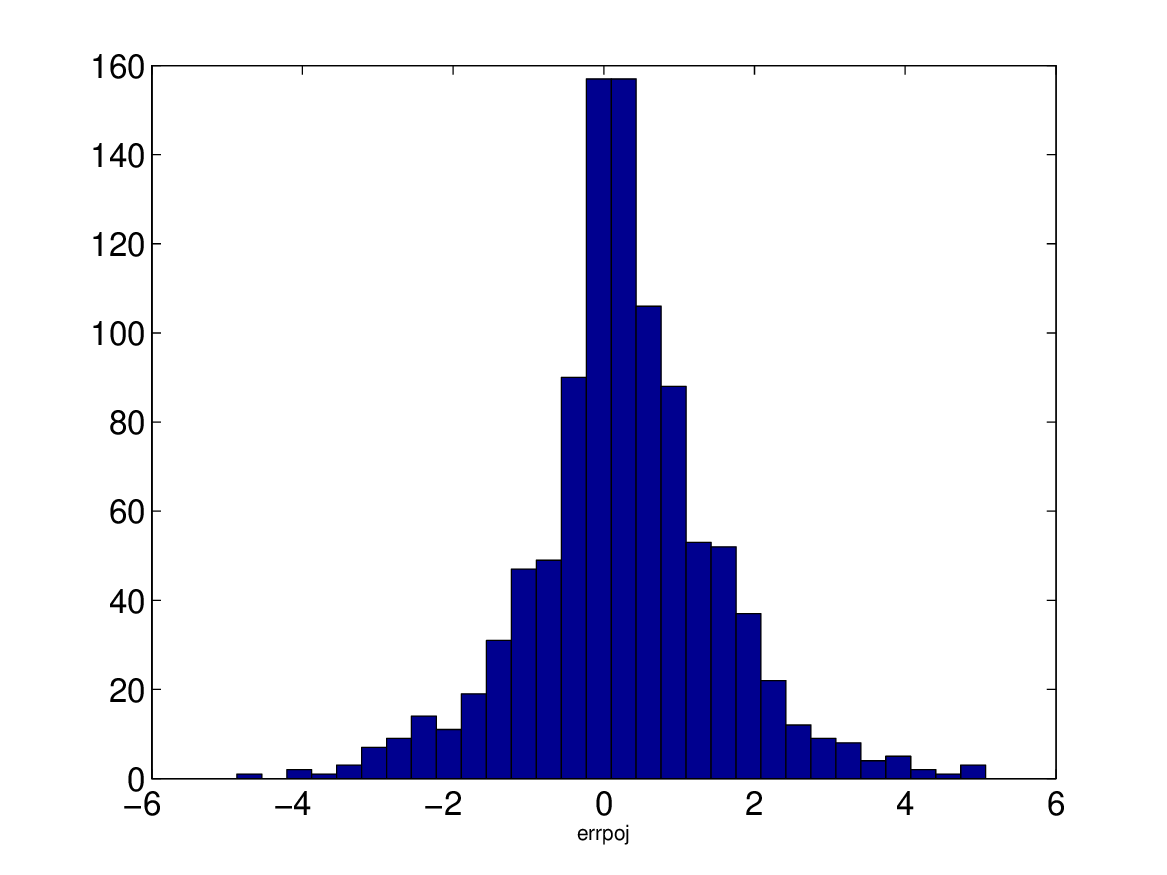}&
\psfrag{errpoj}[t][b]{Projected Error (Bernoulli)}
\includegraphics[width=.45\textwidth]{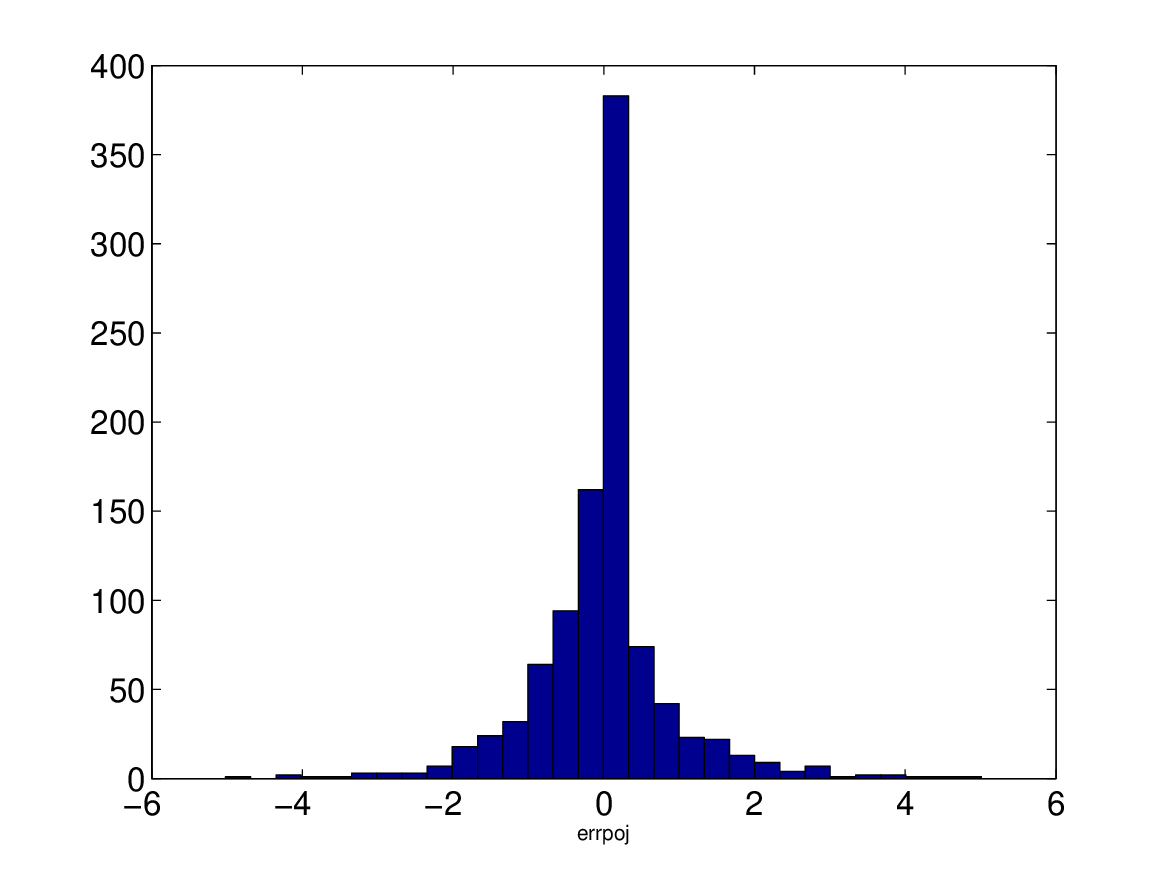}
\end{tabular}
\caption{Projected reconstruction error $v^T(x^\mathrm{lp}-e)$, along a fixed (randomly chosen) direction $v$, using a single Gaussian (left) or Bernoulli (right) design matrix with $p=100$, $m=30$ and a thousand samples of a random sparse signal $e\in\reals^{100}$ with 15 i.i.d. uniform coefficients.\label{fig:projerr}}
\end{center}
\end{figure}

On a random Gaussian matrix with $n=40$ and $m/n=1/2$, we recall in Table~\ref{tab:nap} the recovery threshold $k/m$ certified by the semidefinite relaxation (SDP) detailed in \citep{dAsp08a} and the linear programming (LP) relaxation in \citep{Judi08}, strong and weak recovery thresholds from the asymptotic results in \cite{Dono08}.

\begin{table}[hb]
\begin{center}
\begin{tabular}{|c|c|c|c|}
\hline
SDP & LP & Strong D\&T & Weak D\&T\\
\hline 
0.1 & 0.1 & 0.1 & 0.5\\
\hline
\end{tabular}
\vskip 1ex
\caption{Perfect recovery threshold $k/m$ computed using the semidefinite relaxation (SDP) detailed in \citep{dAsp08a}, the linear programming (LP) relaxation in \citep{Judi08} on a sample Gaussian matrix. We also recall the asymptotic strong and weak recovery thresholds from \cite{Dono08}. \label{tab:nap}}
\end{center}
\end{table}

We then sample Gaussian and Bernoulli matrices of increasing dimensions $n\times n/2$ and plot the mean values of the relaxation bounds on $L(F)$ (blue circles), $\sigma_k(F)$ (brown diamonds) together with $\sum_{i=1}^p \|F_i\|_2$ (black squares). These quantities are plotted in loglog scale in Figure~\ref{fig:LandF}. As expected, the  norm grows as~$n$ while both $\sigma_k(F)$ and $L(F)$ grow as $\sqrt{n}$. In Figure~\ref{fig:Gaussian-scale} we plot the empirical (brown squares) versus predicted (blue circles) probability of recovering signals~$e$, where $F\in\reals^{n\times m}$ is a Gaussian with $n=300$ and $m=n/2$, for various values of the relative cardinality~$k/m$. The empirical probability was obtained by solving~\refp{eq:min-ell1} over one hundred random sparse signal $e\in\reals^{100}$ with 15 i.i.d. uniform coefficients. The predicted probability is obtained by computing $\beta$ from condition~\refp{eq:weak-cond} after bounding $L(F)$ and $\sigma_k(F)$ using the convex relaxations detailed in Section~\ref{s:bounds}.

\begin{figure}[ht]
\begin{center}
\begin{tabular}{cc}
\psfrag{n}[t][b]{Leading dimension $n$}
\includegraphics[width=.45\textwidth,height=.35\textwidth]{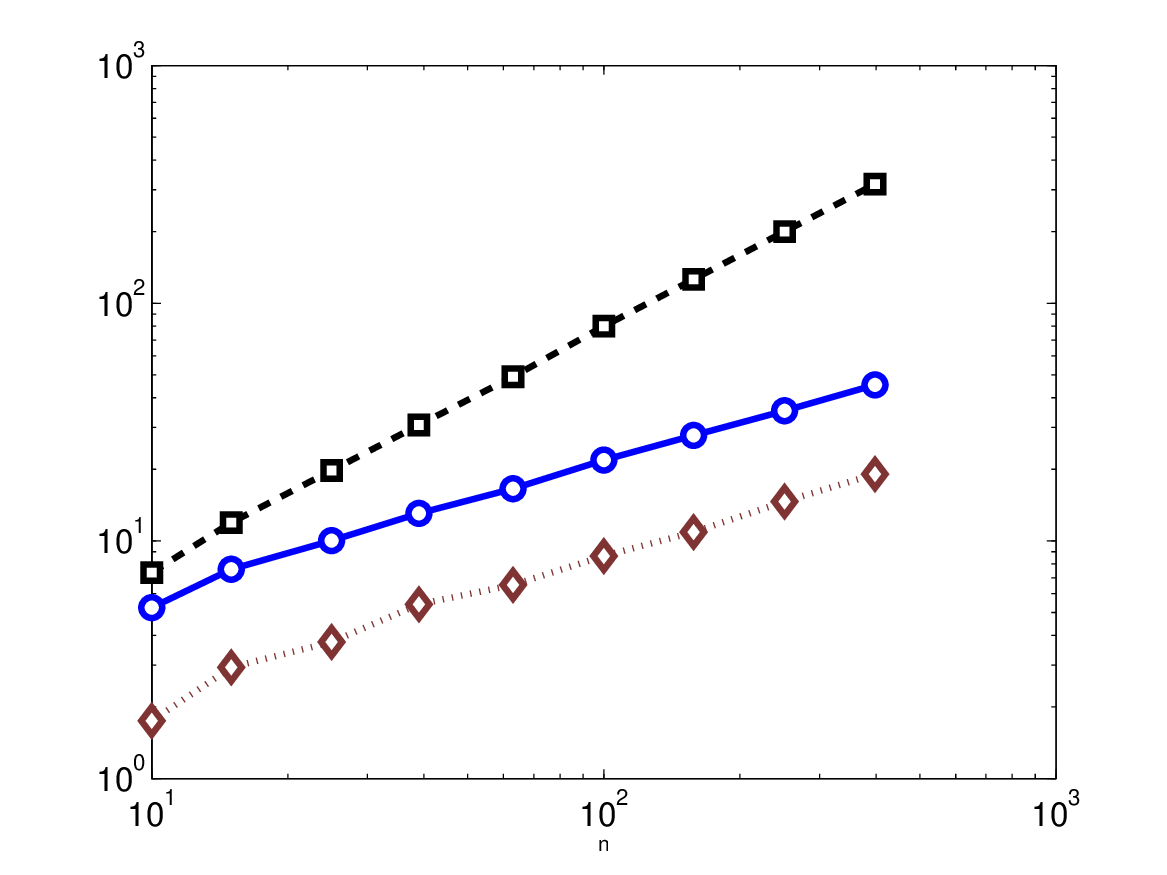}&
\psfrag{n}[t][b]{Leading dimension $n$}
\includegraphics[width=.45\textwidth,height=.35\textwidth]{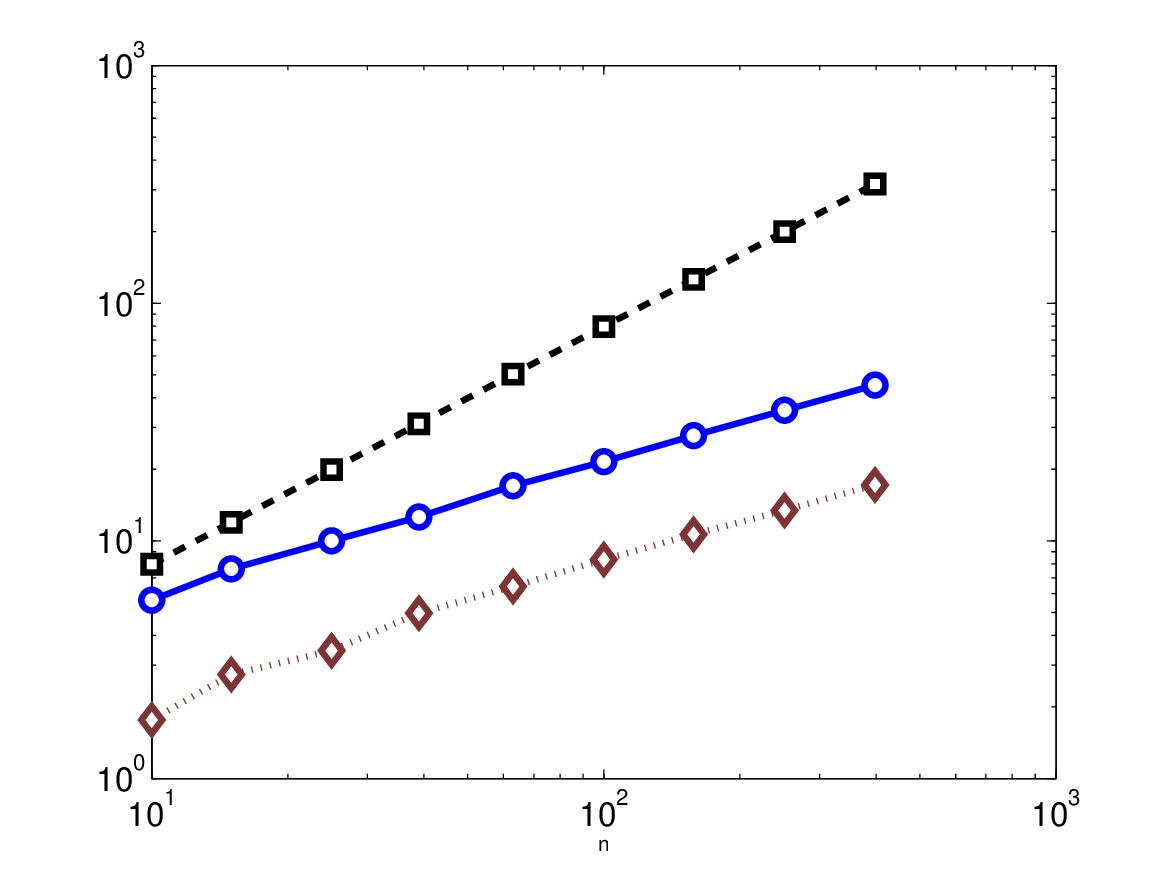}
\end{tabular}
\caption{{\em Left:} Loglog plot of mean values of $L(F)$ (blue circles), $\sigma_k(F)$ (brown diamonds) and $\sum_{i=1}^n \|F_i\|_2$ (black squares) for Gaussian (left) or Bernoulli (right) matrices of increasing dimensions $n$, with $m=n/2$. 
\label{fig:LandF}}
\end{center}
\end{figure}

\begin{figure}[ht]
\begin{center}
\begin{tabular}{cc}
\psfrag{kmratio}[t][b]{Relative Cardinality $k/m$}
\psfrag{probrec}[b][t]{Probability of recovering $e$}
\includegraphics[width=.45\textwidth,height=.36\textwidth]{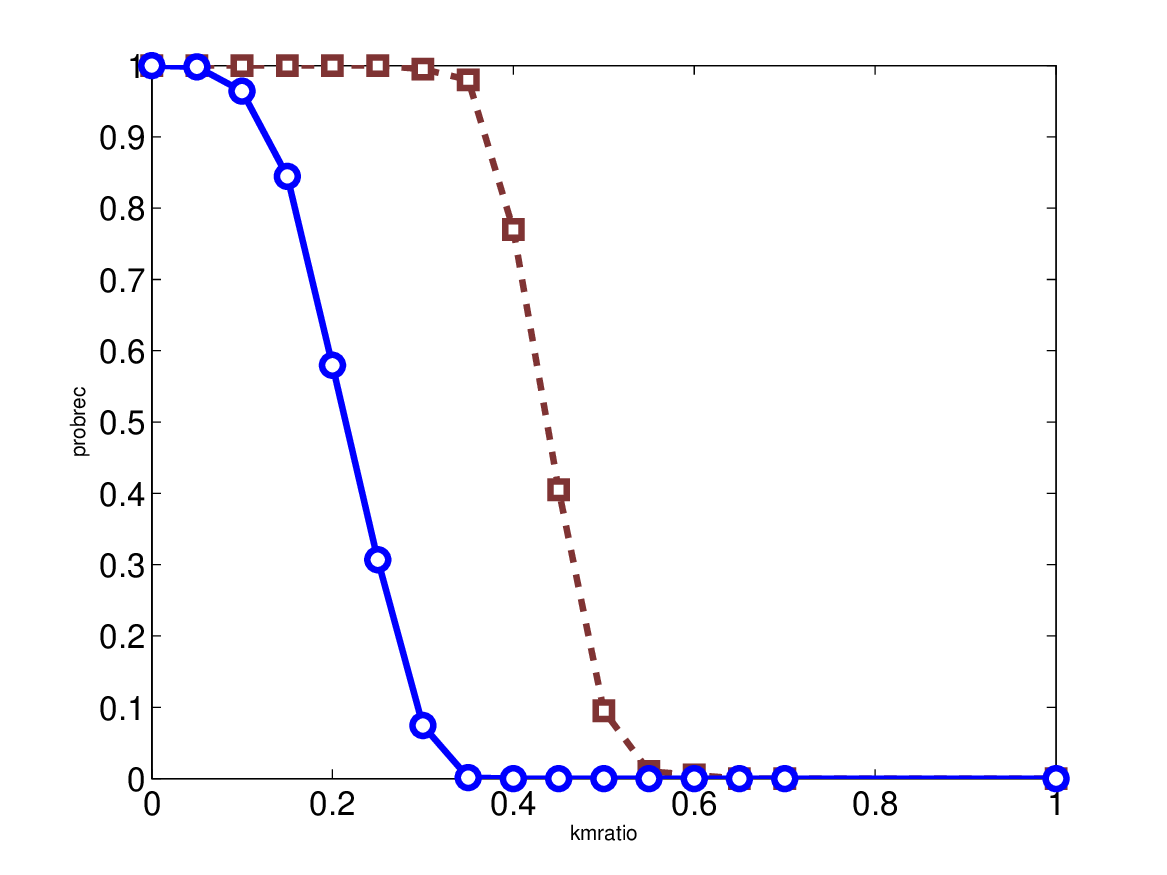}&
\psfrag{kmratio}[t][b]{Relative Cardinality $k/m$}
\psfrag{probrec}[b][t]{Probability of recovering $e$}
\includegraphics[width=.45\textwidth,height=.36\textwidth]{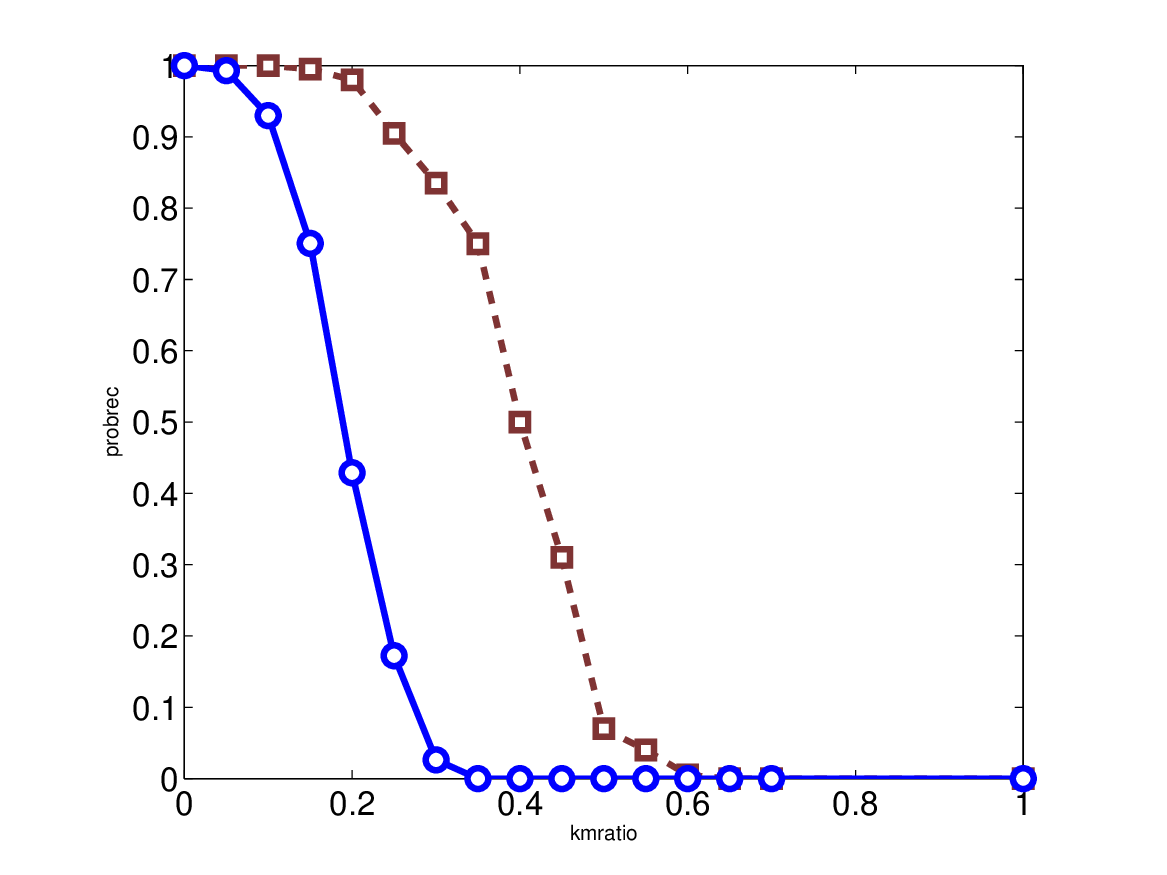}
\end{tabular}
\caption{Empirical (brown squares) versus predicted (blue circles) probability of recovering the true signal~$e$, where $F\in\reals^{n\times m}$ is Gaussian (left) or Bernoulli (right) with $m=n/2$, for various values of the relative cardinality~$k/m$.
\label{fig:Gaussian-scale}}
\end{center}
\end{figure}

\section*{Acknowledgments} The first author would like to acknowledge partial support from NSF grants SES-0835550 (CDI), CMMI-0844795 (CAREER), CMMI-0968842, a starting grant for the European Research Council (project SIPA), a Peek junior faculty fellowship, a Howard B. Wentz Jr. award and a gift from Google. The work of the second author is partially supported by an Alfred P. Sloan research fellowship and NSF grants DMS-0605169 and DMS-0847647 (CAREER).

\section{Appendix}\label{s:appendix}
Gaussian matrices are known to satisfy the recovery condition (\ref{eq:ineq-a}) with high probability for near-optimal values of $k$ hence obviously satisfy (\ref{eq:ineq-rnd}). Here we directly verify that these matrices satisfy condition~(\ref{eq:weak-cond}) w.h.p. without using RIP. Concentration inequalities have been used in \cite{Bara07} to derive a simple proof that some classes of random matrices satisfy RIP, we use similar techniques on the weak recovery property~(\ref{eq:weak-cond}) here.

We start by bounding the fluctuations of the right hand side of inequality (\ref{eq:weak-cond}) when $F\in\reals^{n \times m}$ is a Gaussian random matrix with $F_{ij}\sim{\mathcal N}(0,1/m)$.
\begin{lemma}\label{eq:chi-expect}
Let $F\in\reals^{n \times m}$ with i.i.d $F_{ij}\sim{\mathcal N}(0,1/m)$,
\[
\sum_{i=1}^n\Expect\left[\|F_i\|_2\right]=n(1+O(m^{-1}))
\]
as $m\rightarrow \infty$.
\end{lemma}
\begin{proof}
In this setting, each $\sqrt m\|F_i\|_2$ is $\chi$ distributed with $m$ degrees of freedom, so
\[
\Expect[\|F_i\|_2]=\sqrt{\frac{2}{m}}\frac{\Gamma((m+1)/2)}{\Gamma(m/2)}, \quad i=1,\ldots,n.
\]
Using Stirling's formula \citep[\S6.1.37]{Abra70a}, we get
\BEAS
\frac{\Gamma((m+1)/2)}{\Gamma(m/2)} & = &
\frac{\exp{\left(\frac{m+1}{2}\log\left(\frac{m+1}{2}\right)-\frac{m}{2}\log\left(\frac{m}{2}\right)\right)}}{\sqrt{e(1+1/m)}}+ O(m^{-1/2})\\
& = & \sqrt{\frac{m+1}{2}} + O(m^{-1/2})\\
\EEAS
as $m\rightarrow \infty$, which is the desired result.
\end{proof}

We now use concentration inequalities to bound $\sum_{i=1}^n \|F_i\|_2$ in condition (\ref{eq:weak-cond}) with high probability when $F_{ij}\sim{\mathcal N}(0,1/m)$.
\begin{lemma}\label{eq:cnct-euc}
Let $F\in\reals^{n \times m}$ with i.i.d $F_{ij}\sim{\mathcal N}(0,1/m)$,
\[
\Prob\left[\sum_{i=1}^n \|F_i\|_2  \leq \sum_{i=1}^n \Expect\left[\|F_i\|_2\right] - x  \right] \leq e^{-\frac{mx^2}{2n}}
\]
\end{lemma}
\begin{proof}
For any $U,V\in\reals^{m\times n}$, we have
\BEAS
\sum_{i=1}^n \|U_i\|_2-\|V_i\|_2 & \leq & \sum_{i=1}^n \|U_i-V_i\|_2\\
& \leq & \sqrt{n} \|U-V\|_F
\EEAS
so $\sum_{i=1}^n \|F_i\|_2$ is a $\sqrt{n/m}$-Lipschitz function (w.r.t. the Euclidean norm)
of $nm$ i.i.d. Gaussian variables $F_{ij}/\sqrt{m} \sim{\mathcal N}(0,1)$ and \cite[Th. 3.4]{Mass07} yields the desired result.
\end{proof}

We now turn to the left-hand side of inequality (\ref{eq:weak-cond}) and produce inequalities on $\sigma_k(F)$, using again the fact that it is a Lipschitz function of $F$.
\begin{lemma}\label{eq:cnct-sigma}
Let $F\in\reals^{n \times m}$ with i.i.d. $F_{ij}\sim{\mathcal N}(0,1/m)$,
\[
\Prob\left[ \sigma_k(F) \geq \Expect[\sigma_k(F)] + x \right] \leq e^{-\frac{mx^2}{2k}}
\]
\end{lemma}
\begin{proof}
We first note that the $\max$ is 1-Lispchitz with respect to the $\ell_\infty$ norm on $\reals^n$. Indeed, if $a,b\in\reals^n$
\[
|\max_i a_i - \max_j b_j|\leq \max_{i}|a_i-b_i|,
\]
because
\[
a_i-\max_j b_j\leq a_i-b_i \leq |a_i-b_i|\leq \max_i |a_i-b_i|,\quad i=1,\ldots,n.
\]
Hence, $\max_i a_i-\max_j b_j\leq \max_i |a_i-b_i|$. The two sequences play symmetric roles so we also have $|\max_j b_j -\max_i a_i|\leq \max_k |a_k-b_k|$. Now our aim is to show that $F\rightarrow \sigma_k(F)$ is a Lipschitz function of $F$ with respect to the Euclidian norm. The argument we just gave shows that if $F$ and $G$ are two matrices,
\[
\left|\sigma_k(F)-\sigma_k(G)\right|\leq \max_{\{(u_+,u_-)\in\{0,1\}^{2n},\ones^Tu\leq k\}} \left|\left\|(u_+-u_-)^TF\right\|_2-\left\|(u_+-u_-)^TG\right\|_2\right|,
\]
because $\sigma_k(F)$ and $\sigma_k(G)$ are maxima of finite sequences. We now have
\BEAS
\left|\left\|(u_+-u_-)^TF\right\|_2-\left\|(u_+-u_-)^TG\right\|_2\right| & \leq & \left\|(u_+-u_-)^T(F-G)\right\|_2 \\
& \leq & \|(F-G)\|_2 \left\|(u_+-u_-)^T\right\|_2 \\
& \leq  & \|F-G\|_F \left\|(u_+-u_-)^T\right\|_2
\EEAS
which shows that
\[
\sigma_k(F)=\max_{\{(u_+,u_-)\in\{0,1\}^{2n},\ones^Tu\leq k\}}\left\|(u_+-u_-)^TF\right\|_2
\]
is a Lipschitz function of the entries of $F$ (with respect to Euclidian norm). Now when the entries of $F$ are i.i.d ${\mathcal N}(0,1/m)$, $\sigma_k(F)$ is a Lipschitz function of standard Gaussian variables with Lipschitz constant
\[
\max_{\{(u_+,u_-)\in\{0,1\}^{2n},\ones^Tu\leq k\}}\frac{\left\|u_+-u_-\right\|_2}{\sqrt{m}} \leq \sqrt{\frac{k}{m}},
\]
and \cite[Th. 3.12]{Mass07} yields the desired result.
\end{proof}

Next, to bound $\Expect[\sigma_k(F)]$, we first show a bound on the supremum of an arbitrary number of $\chi$ distributed random variables.
\begin{lemma}\label{lem:chi-mean}
Let $\{y_i\}_{i\in T}$ be $\chi$ distributed variables with $m$ degrees of freedom, then
\[
\Expect[\sup_{i\in T} y_i] \leq \sqrt{2 \log |T|}+\frac{\sqrt{2}~\Gamma((m+1)/2)}{\Gamma(m/2)}\leq \sqrt{2 \log |T|}+\sqrt{m}\;.
\]
\end{lemma}
We note that the proof we present applies non only to $\chi$ distributed random variables but more generally to Lipschitz functions of i.i.d normal random variables.
\begin{proof}
Since $y_i$'s have the same mean, we have
$$
\sup_{i\in T} y_i= \Expect[y_i]+\sup_{i\in T} (y_i-\Expect[y_i])\;.
$$
Here we know that $\Expect[y_i]=\frac{\sqrt{2}~\Gamma((m+1)/2)}{\Gamma(m/2)}$ and we also know using Jensen's inequality that $\Expect[y_i]\leq \sqrt{\Expect[y_i^2]}=\sqrt{m}$.

The fact that a standard multivariate normal satisfies a log-Sobolev inequality (with constant 1 in the setup of \citet[Chap.\,5]{Ledo05}) implies through the Herbst argument that any 1-Lipschitz function $F$ (with respect to Euclidian norm) of i.i.d Gaussian random variables satisfies (see \citet[Eq.5.8]{Ledo05})
$$
\log \Psi(z)\triangleq \log \Expect[\exp\{z(F(X)-\Expect[F(X)])\}]\leq \frac{z^2}{2}\;.
$$
The previous inequality naturally applies to $y_i$'s since a $\chi_m$ random variable is just the norm of a $m$-dimensional vector with i.i.d entries (and the norm is 1-Lipschitz by the triangle inequality).

Using a classic approach in probability, namely a ``soft-max'' inequality, the concavity of the $\log$, the definition of $\Psi(z)$ and the fact that the variables $y_i$ are identically distributed, we now have, if $\tilde{y}_i=y_i-\Expect[y_i]$,
\BEAS
\Expect[\sup_{i\in T} \tilde{y}_i] & \leq & \frac{1}{z} \Expect\left[\log\left(  \sum_{i\in T}  e^{z \tilde{y}_i}\right)\right]\\
& \leq & \frac{1}{z} \log\left(  \sum_{i\in T} \Expect\left[ e^{z \tilde{y}_i}\right]\right)\\
& \leq & \frac{\log|T|+\log \Psi(z)}{z} \leq \frac{\log|T|+z^2/2}{z}
\EEAS
for any $z>0$. Optimizing over $z$, we get that
$$
\Expect[\sup_{i\in T} \tilde{y}_i]\leq \sqrt{2\log |T|}\;.
$$
This gives the desired result.
\end{proof}

Let us now assume that the basis $F\in\reals^{n \times m}$ is a Gaussian random matrix (hence $A$ is implicitly defined here as a matrix annihilating $F$ on the left) with $F_{ij}\sim{\mathcal N}(0,1/m)$. As detailed below and throughout this appendix, standard concentration arguments allow us to directly show that $F$ satisfies condition~(\ref{eq:weak-cond}), without resorting to the restricted isometry property. We assume that $m$ scales proportionally to $n$, with $m=\mu n$ as $n$ goes to infinity. We also assume that $k$ scales as $\kappa m u_m$ with $u_m \rightarrow 0$ when $m$ and $n$ go to infinity.
\begin{proposition}\label{prop:gaussian-ok}
Suppose $m=\mu n$ and $k=\kappa m u_m$ for some $\mu,\kappa\in(0,1)$, with $u_m\rightarrow 0$ as $m\rightarrow \infty$. Let $F\in\reals^{n \times m}$ be an i.i.d. Gaussian random matrix with $F_{ij}\sim{\mathcal N}(0,1/m)$ and $\beta>0$, then $F$ satisfies condition (\ref{eq:weak-cond}) with high probability as $n$ goes to infinity.
\end{proposition}
\begin{proof}
We first study the left hand side of (\ref{eq:weak-cond}), which reads
\[
\left(\sqrt{2k\left(1+\log\frac{2n}{k}\right)}+\beta \right)\sigma_k(F) \leq  \left(\sqrt{\frac{2}{\pi}} \sum_{i=1}^n \|F_i\|_2 - \beta L(F)  \right) \alpha_k,
\]
when $n$ goes to infinity. Because $\sqrt{m/k}\left\|(u_+-u_-)^TF\right\|_2$ is $\chi$ distributed with $m$ d.f. whenever $u=(u_+,u_-)\in\{0,1\}^{2n}$ with $\ones^Tu=k$ and $u_+^Tu_-=0$, Lemma \ref{lem:chi-mean} shows that for $n$ large enough
\BEAS
\Expect[\sigma_k(F)]&=&\Expect\left[\max_{\{u=(u_+,u_-)\in\{0,1\}^{2n},\ones^Tu\leq k\}}\left\|(u_+-u_-)^TF\right\|_2\right]\\
&=&\Expect\left[\max_{\{u=(u_+,u_-)\in\{0,1\}^{2n},\ones^Tu= k,u_+^Tu_-=0\}}\left\|(u_+-u_-)^TF\right\|_2\right]\\
&\leq& \sqrt{\frac{k}{m}} \left[\sqrt{2k\left(1+\log\left(\frac{2n}{k}\right)\right)}+\sqrt{m}\right]\;.
\EEAS
Here we have used the fact that the cardinality of the set $T$ over which we are taking a supremum is such that
$\log|T|\leq k\left(1+\log\left(\frac{2n}{k}\right)\right)$, as shown in the proof of Lemma~\ref{lem:cond-left}.
We note that for a constant $c>0$, we have $k\log\left(\frac{2n}{k}\right)\leq c m u_m \log(1/u_m)\ll m$.
Therefore, if $c$ denotes a constant that may change from display to display (but does not depend on $n$ or $m$), we have
$$
\Expect[\sigma_k(F)]\leq c \sqrt{k}\;,
$$
and
\[
\left(\sqrt{2k\log\frac{2n}{k}} \right)\frac{\Expect[\sigma_k(F)]}{m} \leq c \frac{k}{m}\sqrt{- \log(u_m)}\leq c \sqrt{- u_m^2 \log(u_m)}\xrightarrow[n \rightarrow \infty]{}0
\]
where $c>0$ does not depend on $n$. For some arbitrarily small $\nu>0$, setting $x=n^\nu\sqrt{2k/m}$ in Lemma \ref{eq:cnct-sigma}, yields
\[
\Prob\left[ {\sigma_k(F)} \geq \Expect[\sigma_k(F)] + n^\nu \sqrt{{2k}/{m}} \right] \leq e^{-n^{2\nu}}.
\]
We now focus on the right hand side of (\ref{eq:weak-cond}). Lemma \ref{eq:chi-expect} shows that
\[
\lim_{n \rightarrow \infty} \frac{\sum_{i=1}^n\Expect\left[\|F_i\|_2\right]}{m}=\lim_{n \rightarrow \infty} \frac{n}{m}=\frac{1}{\mu}.
\]
because $\sqrt{m}\|F_i\|_2$ is $\chi$ distributed with $m$ degrees of freedom. Setting $x^2=n^{\nu+1}/m$ in Lemma \ref{eq:cnct-euc} then yields
\[
\Prob\left[\frac{\sum_{i=1}^n \|F_i\|_2}{m}  \leq \frac{\sum_{i=1}^n \Expect\left[\|F_i\|_2\right]}{m} - \frac{n^{\nu+1/2}\sqrt{2}}{m^{3/2}}  \right] \leq e^{-n^{2\nu}}
\]
which, together with the inequality on the left hand side derived above, means that for $n$ large enough, the matrix $F$ satisfies condition~(\ref{eq:weak-cond}) with probability at least $1-2e^{-n^{2\nu}}$. Finally, with $L(F)^2\leq n \|FF^T\|_2$, the fact that $\|F\|_2$ is 1-Lipschitz (with respect to Euclidian norm as a function of the (Gaussian) entries of $F$) combined with the bound on $\Expect[\|F\|_2]$ detailed in \citet[Prop.\,2.14]{Davi01} shows that
\[
\Prob\left[\|F\|_2  \leq c + \sqrt{2}n^\nu \right] \leq e^{-n^{2\nu}},
\]
for some absolute constant $c>0$. This means that $L(F)/m \rightarrow 0$ when $n$ goes to infinity and the second term of the right-hand side of (\ref{eq:weak-cond}) is then negligible compared to the first.
\end{proof}

This last result shows that the sufficient condition in (\ref{eq:weak-cond}) is weak enough on Gaussian matrices to hold w.h.p. near optimal values of the cardinality where the number of samples $m$ is almost proportional to the number of nonzero coefficients in the signal.

\small{\bibliographystyle{plainnat}
\bibsep 1ex
\bibliography{MainPerso}}

\begin{thebibliography}{45}
\providecommand{\natexlab}[1]{#1}
\providecommand{\url}[1]{\texttt{#1}}
\expandafter\ifx\csname urlstyle\endcsname\relax
  \providecommand{\doi}[1]{doi: #1}\else
  \providecommand{\doi}{doi: \begingroup \urlstyle{rm}\Url}\fi

\bibitem[Abramowitz and Stegun(1970)]{Abra70a}
M.~Abramowitz and I.~Stegun.
\newblock \emph{Handbook of Mathematical Functions}.
\newblock Dover, New York, 1970.

\bibitem[Affentranger and Schneider(1992)]{Affe92}
F.~Affentranger and R.~Schneider.
\newblock {Random projections of regular simplices}.
\newblock \emph{Discrete and Computational Geometry}, 7\penalty0 (1):\penalty0
  219--226, 1992.

\bibitem[Arora et~al.(1995)Arora, Karger, and Karpinski]{Aror95}
S.~Arora, D.~Karger, and M.~Karpinski.
\newblock {Polynomial time approximation schemes for dense instances of graph
  problems}.
\newblock In \emph{Proc. of 28th STOC}, pages 193--210, 1995.

\bibitem[Baraniuk et~al.(2008)Baraniuk, Davenport, DeVore, and Wakin]{Bara07}
R.~Baraniuk, M.~Davenport, R.~DeVore, and M.~Wakin.
\newblock {A simple proof of the restricted isometry property for random
  matrices}.
\newblock \emph{Constructive Approximation}, 28\penalty0 (3):\penalty0
  253--263, 2008.

\bibitem[Ben-Tal et~al.(2009)Ben-Tal, El~Ghaoui, and Nemirovski]{Ben09}
A.~Ben-Tal, L.~El~Ghaoui, and A.S. Nemirovski.
\newblock \emph{{Robust optimization}}.
\newblock Princeton University Press, 2009.

\bibitem[Bickel et~al.(2007)Bickel, Ritov, and Tsybakov]{Bick07}
P.~Bickel, Y.~Ritov, and A.~Tsybakov.
\newblock Simultaneous analysis of lasso and dantzig selector.
\newblock \emph{Preprint Submitted to the Annals of Statistics}, 2007.

\bibitem[Billionnet and Roupin(2006)]{Bill06}
A.~Billionnet and F.~Roupin.
\newblock {A deterministic approximation algorithm for the densest k-subgraph
  problem}.
\newblock \emph{International Journal of Operational Research}, 2006.

\bibitem[Candes and Tao(2007)]{Cand07}
E.~Candes and T.~Tao.
\newblock {The Dantzig selector: statistical estimation when p is much larger
  than n}.
\newblock \emph{Annals of Statistics}, 35\penalty0 (6):\penalty0 2313--2351,
  2007.

\bibitem[Cand\`es and Tao(2005)]{Cand05}
E.~J. Cand\`es and T.~Tao.
\newblock Decoding by linear programming.
\newblock \emph{IEEE Transactions on Information Theory}, 51\penalty0
  (12):\penalty0 4203--4215, 2005.

\bibitem[Cand\`es and Plan(2009)]{Cand09a}
E.J. Cand\`es and Y.~Plan.
\newblock Near-ideal model selection by $\ell_1$ minimization.
\newblock \emph{Annals of Statistics}, 37:\penalty0 2145--2177, 2009.

\bibitem[Cand\`es and Tao(2006)]{Cand06}
E.J. Cand\`es and T.~Tao.
\newblock Near-optimal signal recovery from random projections: Universal
  encoding strategies?
\newblock \emph{IEEE Transactions on Information Theory}, 52\penalty0
  (12):\penalty0 5406--5425, 2006.

\bibitem[Cohen et~al.(2009)Cohen, Dahmen, and DeVore]{Cohe06}
A.~Cohen, W.~Dahmen, and R.~DeVore.
\newblock {Compressed sensing and best k-term approximation}.
\newblock \emph{Journal of the {AMS}}, 22\penalty0 (1):\penalty0 211--231,
  2009.

\bibitem[d'Aspremont et~al.(2008)d'Aspremont, Bach, and El~Ghaoui]{dAsp08b}
A.~d'Aspremont, F.~Bach, and L.~El~Ghaoui.
\newblock Optimal solutions for sparse principal component analysis.
\newblock \emph{Journal of Machine Learning Research}, 9:\penalty0 1269--1294,
  2008.

\bibitem[d'Aspremont and El~Ghaoui(2011)]{dAsp08a}
Alexandre d'Aspremont and Laurent El~Ghaoui.
\newblock Testing the nullspace property using semidefinite programming.
\newblock \emph{Mathematical Programming}, 127:\penalty0 123--144, 2011.

\bibitem[Davidson and Szarek(2001)]{Davi01}
K.R. Davidson and S.J. Szarek.
\newblock Local operator theory, random matrices and banach spaces.
\newblock \emph{Handbook of the geometry of Banach spaces}, 1:\penalty0
  317--366, 2001.

\bibitem[Donoho(2004)]{Dono04}
D.~L. Donoho.
\newblock Neighborly polytopes and sparse solution of underdetermined linear
  equations.
\newblock \emph{Stanford dept. of statistics working paper}, 2004.

\bibitem[Donoho and Tanner(2005)]{Dono05}
D.~L. Donoho and J.~Tanner.
\newblock Sparse nonnegative solutions of underdetermined linear equations by
  linear programming.
\newblock \emph{Proc. of the National Academy of Sciences}, 102\penalty0
  (27):\penalty0 9446--9451, 2005.

\bibitem[Donoho(2006)]{Dono06a}
D.L. Donoho.
\newblock {Compressed sensing}.
\newblock \emph{IEEE Transactions on Information Theory}, 52\penalty0
  (4):\penalty0 1289--1306, 2006.

\bibitem[Donoho and Huo(2001)]{Dono01}
D.L. Donoho and X.~Huo.
\newblock {Uncertainty principles and ideal atomic decomposition}.
\newblock \emph{IEEE Transactions on Information Theory}, 47\penalty0
  (7):\penalty0 2845--2862, 2001.

\bibitem[Donoho and Tanner(2008)]{Dono08}
D.L. Donoho and J.~Tanner.
\newblock {Counting the Faces of Randomly-Projected Hypercubes and Orthants,
  with Applications}.
\newblock \emph{Arxiv preprint arXiv:0807.3590}, 2008.

\bibitem[Feige and Langberg(2001)]{Feig01}
U.~Feige and M.~Langberg.
\newblock {Approximation algorithms for maximization problems arising in graph
  partitioning}.
\newblock \emph{Journal of Algorithms}, 41\penalty0 (2):\penalty0 174--211,
  2001.

\bibitem[Feige and Seltser(1997)]{Feig97}
U.~Feige and M.~Seltser.
\newblock {On the densest $k$-subgraph problem}.
\newblock Technical report, Department of Applied Mathematics and Computer
  Science, The Weizmann Institute, 1997.

\bibitem[Feige et~al.(2001)Feige, Peleg, and Kortsarz]{Feig01a}
U.~Feige, D.~Peleg, and G.~Kortsarz.
\newblock {The dense $k$-subgraph problem}.
\newblock \emph{Algorithmica}, 29\penalty0 (3):\penalty0 410--421, 2001.

\bibitem[Goemans and Williamson(1995)]{goem95}
M.X. Goemans and D.P. Williamson.
\newblock Improved approximation algorithms for maximum cut and satisfiability
  problems using semidefinite programming.
\newblock \emph{J. ACM}, 42:\penalty0 1115--1145, 1995.

\bibitem[Han et~al.(2002{\natexlab{a}})Han, Ye, and Zhang]{Han02}
Q.~Han, Y.~Ye, and J.~Zhang.
\newblock {An improved rounding method and semidefinite programming relaxation
  for graph partition}.
\newblock \emph{Mathematical Programming}, 92\penalty0 (3):\penalty0 509--535,
  2002{\natexlab{a}}.

\bibitem[Han et~al.(2002{\natexlab{b}})Han, Ye, and Zhang]{Han02a}
Q.~Han, Y.~Ye, and J.~Zhang.
\newblock {An improved rounding method and semidefinite programming relaxation
  for graph partition}.
\newblock \emph{Mathematical Programming}, 92\penalty0 (3):\penalty0 509--535,
  2002{\natexlab{b}}.
\newblock ISSN 0025-5610.

\bibitem[Helmberg et~al.(2000)Helmberg, Rendl, and Weismantel]{Helm00a}
C.~Helmberg, F.~Rendl, and R.~Weismantel.
\newblock {A semidefinite programming approach to the quadratic knapsack
  problem}.
\newblock \emph{Journal of Combinatorial Optimization}, 4\penalty0
  (2):\penalty0 197--215, 2000.

\bibitem[Juditsky and Nemirovski(2011)]{Judi08}
A.~Juditsky and A.S. Nemirovski.
\newblock On verifiable sufficient conditions for sparse signal recovery via
  $\ell_1$ minimization.
\newblock \emph{Mathematical Programming Series B}, 127\penalty0 (57-88), 2011.

\bibitem[Kortsarz and Peleg(1993)]{Kort93}
G.~Kortsarz and D.~Peleg.
\newblock {On choosing a dense subgraph}.
\newblock In \emph{Foundations of Computer Science, 1993. Proceedings., 34th
  Annual Symposium on}, pages 692--701, 1993.

\bibitem[Ledoux(2005)]{Ledo05}
M.~Ledoux.
\newblock \emph{{The Concentration of Measure Phenomenon}}.
\newblock American Mathematical Society, 2005.

\bibitem[Lin(1998)]{Lin98}
E.Y.H. Lin.
\newblock {A Biblographical Survey on Some Wellknown Non-Standard Knapsack
  Problems}.
\newblock \emph{Information Systems and Operational Research}, 36\penalty0
  (4):\penalty0 274--317, 1998.

\bibitem[Massart(2007)]{Mass07}
P.~Massart.
\newblock Concentration inequalities and model selection.
\newblock \emph{Ecole d'Et\'e de Probabilit\'es de Saint-Flour XXXIII}, 2007.

\bibitem[Meinshausen and Yu(2008)]{Mein07}
N.~Meinshausen and B.~Yu.
\newblock {Lasso-type recovery of sparse representations for high-dimensional
  data}.
\newblock \emph{Annals of Statistics}, 37\penalty0 (1):\penalty0 246--270,
  2008.

\bibitem[Meinshausen et~al.(2007)Meinshausen, Rocha, and Yu]{Mein07a}
N.~Meinshausen, G.~Rocha, and B.~Yu.
\newblock A tale of three cousins: Lasso, l2boosting, and danzig.
\newblock \emph{Annals of Statistics}, 35\penalty0 (6):\penalty0 2373--2384,
  2007.

\bibitem[Nemirovski(2001)]{Nemi01}
A.S. Nemirovski.
\newblock The matrix cube problem: Approximations and applications.
\newblock \emph{INFORMS}, 2001.

\bibitem[Nemirovski(2005)]{Stei05}
A.S. Nemirovski.
\newblock \emph{{Computation of matrix norms with applications to Robust
  Optimization}}.
\newblock PhD thesis, Technion, 2005.

\bibitem[Nesterov(1998{\natexlab{a}})]{Nest98a}
Y.~Nesterov.
\newblock \emph{Global quadratic optimization via conic relaxation}.
\newblock Number 9860. CORE Discussion Paper, 1998{\natexlab{a}}.

\bibitem[Nesterov(1998{\natexlab{b}})]{Nest98b}
Y.~Nesterov.
\newblock {Semidefinite relaxation and nonconvex quadratic optimization}.
\newblock \emph{Optimization methods and software}, 9\penalty0 (1):\penalty0
  141--160, 1998{\natexlab{b}}.

\bibitem[Nesterov(2007)]{Nest04a}
Y.~Nesterov.
\newblock Smoothing technique and its applications in semidefinite
  optimization.
\newblock \emph{Mathematical Programming}, 110\penalty0 (2):\penalty0 245--259,
  2007.

\bibitem[Pisinger(2007)]{Pisi07}
D.~Pisinger.
\newblock {The quadratic knapsack problem---a survey}.
\newblock \emph{Discrete Applied Mathematics}, 155\penalty0 (5):\penalty0
  623--648, 2007.

\bibitem[Srivastav and Wolf(1998)]{Sriv98}
A.~Srivastav and K.~Wolf.
\newblock {Finding dense subgraphs with semidefinite programming}.
\newblock \emph{Lecture Notes in Computer Science}, 1444:\penalty0 181--192,
  1998.

\bibitem[Tibshirani(1996)]{Tibs96}
R.~Tibshirani.
\newblock Regression shrinkage and selection via the {LASSO}.
\newblock \emph{Journal of the {R}oyal statistical society, series {B}},
  58\penalty0 (1):\penalty0 267--288, 1996.

\bibitem[Vershik and Sporyshev(1992)]{Vers92}
AM~Vershik and PV~Sporyshev.
\newblock Asymptotic behavior of the number of faces of random polyhedra and
  the neighborliness problem.
\newblock \emph{Selecta Math. Soviet}, 11\penalty0 (2):\penalty0 181--201,
  1992.

\bibitem[Zhang(2005)]{Zhan05}
Y.~Zhang.
\newblock A simple proof for recoverability of $\ell_1$-minimization: Go over
  or under.
\newblock \emph{Rice University CAAM Technical Report TR05-09}, 2005.

\bibitem[Zhao and Yu(2006)]{Zhao06}
P.~Zhao and B.~Yu.
\newblock On model selection consistency of lasso.
\newblock \emph{Journal of Machine Learning Research}, 7:\penalty0 2541--2563,
  2006.

\end{thebibliography}
\end{document}